\documentclass[12pt, twoside]{article}
\usepackage{amsmath,amsthm,amssymb}
\usepackage{times}
\usepackage{enumerate}
\usepackage{color}
\usepackage{srcltx}

\def\titlerunning#1{\gdef\titrun{#1}}
\makeatletter
\def\author#1{\gdef\autrun{\def\and{\unskip, }#1}\gdef\@author{#1}}

\makeatother

\def\subjclass#1{{\renewcommand{\thefootnote}{}%
\footnote{\emph{Mathematics Subject Classification (2010):} #1}}}
\def\keywords#1{\par\medskip
\noindent\textbf{Keywords.} #1}

\def\CX{{\mathbb C}}
\def\RX{{\mathbb R}}
\def\QX{{\mathbb Q}}
\def\NX{{\mathbb N}}
\def\ZX{{\mathbb Z}}

\def\GL{{\rm GL}}
\def\SL{{\rm SL}}
\def\gl{{\rm gl}}

\def\QED{\hbox{\hskip 1pt \vrule width4pt height 6pt depth 1.5pt \hskip 1pt}}

\def\calM{{\cal M}}
\def\calN{{\cal N}}

\def\calA{{\cal A}}

\newcommand{\ie}{{\it i.e.}}
\newcommand{\diag}{ {\rm diag} }

\renewcommand{\Re}{\re}
\renewcommand{\Im}{\mbox{\rm Im\,}}

\DeclareMathOperator{\re}{Re}

\newtheorem{theorem}{Theorem}
\newtheorem{prop}[theorem]{Proposition}
\newtheorem{cor}[theorem]{Corollary}
\newtheorem{lemma}[theorem]{Lemma}

\newtheorem{remark}[theorem]{Remark}

\def\QED{\hbox{\hskip 1pt \vrule width4pt height 6pt depth 1.5pt \hskip 1pt}}
\newenvironment{prff}[1]{\trivlist
\item[\hskip \labelsep{\bf #1.\hspace*{.3em}}]}{~\hspace{\fill}~$\square$\endtrivlist}
\newenvironment{prf}{
\begin{prff}{Proof}}{
\end{prff}}
 \def\square{\QED}

\begin{document}

\baselineskip=17pt

\titlerunning{Consistent systems}

\title{Consistent systems of linear differential and difference equations}
\author{Reinhard Sch\"afke\footnote{ 
Institut de Recherche Math\'ematique Avanc\'ee,
Universit\'e de  Strasbourg et C.N.R.S.,
7, rue Ren\'e Descartes,
67084 Strasbourg Cedex, France, {\tt schaefke@unistra.fr}.},   
Michael F. Singer\footnote{Department of Mathematics, North Carolina State University, Box 8205, Raleigh, NC 27695, USA, {\tt singer@ncsu.edu}.}}

\date{June 19th, 2017}
\maketitle


\begin{abstract} {We consider  systems of linear differential and difference equations 
\begin{eqnarray*}
\delta Y(x) =A(x)Y(x), \ \ \ \ \ 
\sigma Y(x) =B(x)Y(x) 
\end{eqnarray*}
with $\delta = \frac{d}{dx}$, $\sigma$ a shift operator {$\sigma(x) = x+a$}, 
$q$-{dilation} operator {$\sigma(x) = qx$} or Mahler operator {$\sigma(x) = x^p$}
and systems of two linear difference equations
\begin{eqnarray*}
\sigma_1 Y(x) =A(x)Y(x), \ \ \ \ \ 
\sigma_2 Y(x) =B(x)Y(x) 
\end{eqnarray*}
with $(\sigma_1,\sigma_2)$ a sufficiently independent  pair  of shift operators, pair of $q$-{dilation} operators or  pair of Mahler operators.
Here $A(x)$ and $B(x)$ are $n\times n$ matrices with rational function entries.
Assuming a consistency hypothesis, we show that such systems can be reduced to a system of a very simple form. 
Using this we characterize functions satisfying two linear scalar differential or difference equations with respect to these operators.  We also indicate how these results have consequences both  in the theory of automatic sets, leading to a new proof of Cobham's Theorem, and  in the Galois theories of linear difference and differential equations, leading to hypertranscendence results.}
\end{abstract}

\section{Introduction} In \cite{Ramis92}, J.-P.~Ramis showed  that if a formal power series $f(x)$  is a solution of  a linear differential equation and a linear $q$-difference equation\footnote{A linear difference equation involving the operator $\sigma(x) = qx$}, {$q\neq 0$,
$q$ transcendental if $|q|=1$},
both with polynomial coefficients, then $f$ is the expansion at the origin of a rational function. Rationality has also been shown for a formal power series satisfying
\begin{itemize}
\item a linear differential equation and a linear $\sigma$-difference equation with polynomial coefficients, where $\sigma$ is the Mahler operator $\sigma(f)(x) = f(x^k), k\mbox{ an integer} \geq2$ \cite{Bez94}, 
\item   a linear $q_1$-difference equation and a linear $q_2$-difference equation, both with polynomial coefficients and with $q_1$ and $q_2$ multiplicatively independent ({\it i.e.,} no integer power of $q_1$  is equal to an integer power of $q_2$) \cite{BB92}
\footnote{The result of \cite{BB92} needs some restrictions, see below Corollary \ref{cor2q}.}, 
\item   a linear $\sigma_1$-difference equation and $\sigma_2$-difference equation with polynomial coefficients and with Mahler operators $\sigma_1, \sigma_2$ having multiplicatively independent exponents \cite{AuB16}.  
\end{itemize}
Other results characterizing entire solutions of a linear differential equation and a linear $\sigma$-equation with polynomial coefficients where $\sigma$ is the {operator $\sigma(x)= x+\alpha$, $\sigma(x) = qx$, or $\sigma(x) = x^k$ and entire solutions of two linear $\sigma$-difference equations involving these operators can be found in  \cite{BG96}, \cite{Bezivin00},\cite{BH99}, \cite{Jolly}, \cite{LuPe98}, \cite{Mart},   and \cite{Ran92}.}


    These results have been proved with a variety of ideas such as the structure of ideals of entire functions,  Gevrey-type estimates, $p$-adic behavior and mod $p$ reductions. In our work we present a unified approach to all these results, reproving and generalizing them to also characterize meromorphic solutions on the plane and certain Riemann surfaces.

Our results spring from two fundamental ideas.  The first is that  questions concerning the form of solutions of two scalar linear differential/difference equations can be reduced to showing that consistent pairs of first order systems are equivalent to very simple systems.  {The second is that the hypothesis of consistency allows us to show that the singular points are of a very simple nature, to describe the interaction of local solutions at different singular points and to continue local solutions meromorphically.  These conclusions, in turn, allow us to prove that the systems are equivalent to systems of a very simple form.} {For our approach, it is crucial that $\delta$ and 
$\sigma$  { or $\sigma_1$ and $\sigma_2$, respectively, }
 commute except for some constant factor. The commutativity is closely related to the consistency condition.}

{Our approach is best explained with the example of Ramis's result  in the case $|q|\neq 0,1$.
Let $f(x)$ be a power series satisfying both a linear differential equation and a linear $q$-difference equation with coefficients in $\CX(x)$. 
Using these equations, one shows that the $\CX(x)$-vector space $V$ spanned by $\{(x\frac{d}{dx})^if(q^jx)\}$ with 
$0\leq i<\infty, -\infty <j<\infty$ is finite dimensional (see Corollary~\ref{cor0}).  This space consists of Laurent series and 
is invariant under the map $\sigma$ that sends $x$ to $qx$ and the derivation $\delta=x\frac{d}{dx}$.  
{If $y_1(x), \ldots , y_n(x)$ is a $\CX(x)$-basis of $V$ and
$y(x) = (y_1(x), \ldots , y_n(x))^T$, then}
\begin{equation}\label{eq000}
\begin{aligned}
 x\frac{d y(x)}{d x} &= &A(x)y(x), &  \ \ A(x) \in \gl_n(\CX(x))\\
 y( qx) &=& B(x)y(x), & \ \ B(x) \in \GL_n(\CX(x)).
\end{aligned} 
\end{equation}
$B(x)$ is invertible because $\sigma$ is an automorphism.
Calculating $\sigma(\delta (y(x)))=\delta(\sigma(y(x)))$ in two ways and using that the components of $y$ are linearly independent 
over $\CX(x)$, we obtain that $A(x)$ and $B(x)$  satisfy the consistency condition 
\begin{equation}\label{eq000a}
x\frac{d B(x)}{d x} + B(x) A(x) = A(qx) B(x).
 \end{equation}

The first principal result of our work (Theorem~\ref{thm1}) states in this case that 
there exists a matrix $G(x) \in \GL_n(\CX(x))$ such that the gauge transformation $y(x) = G(x) z(x)$ results in a new  simpler system 
\begin{equation}\label{eq001}
\begin{aligned}
 x\frac{d z(x)}{d x} &= &{\tilde{A}}z(x), \\
z(qx) &=&\tilde{B}z(x)
\end{aligned} 
\end{equation}
with $\tilde{A} \in \gl_n(\CX)$ and $\tilde{B} \in \GL_n(\CX)$.

This implies that there is a new basis $z(x) = (z_1(x), \ldots , z_n(x))^T$ {of $V$} given by $z(x) = G(x)^{-1} y(x)$ such that  
$x\frac{dz(x)}{dx} = \tilde{A}z(x)$ and $z(qx) = \tilde{B}z(x)$ with $\tilde{A}$ and $\tilde{B}$ constant matrices.  It is not hard to 
show that the entries of $z(x)$ must be Laurent polynomials and therefore rational.  We then conclude that $y(x)$ is also rational 
and hence also the given $f(x)$ {is rational}.

{We now give an idea of the proof of Theorem~\ref{thm1} in the context of the  present case. A calculation shows that the 
consistency condition implies: {\it 
if $Y(x)$ is a solution of  $x\frac{d Y(x)}{d x} = A(x)Y(x)$
 then $Z(x) = B(x) Y(x)$ is a solution of $x\frac{dZ(x)}{dx} = A(qx)Z(x)$}.}
 Repeating this observation we have that for any $m$, there is a gauge transformation $Y(x) = D_m(x)Z(x)$ taking solutions of $x\frac{d Y(x)}{d x} = A(x)Y(x)$ to solutions of $x\frac{dZ(x)}{dx} = A(q^mx)Z(x)$.  This gauge transformation only introduces  apparent singularities, that is, those at which one has a meromorphic fundamental solution matrix. Since the singular points in $\CX\backslash \{0,\infty\}$ of these two equations   are disjoint for sufficiently large $m$, we can conclude that all the singular points, other than $0, \infty$, of $\frac{d Y(x)}{d x} = A(x)Y(x)$  are apparent (see  Lemma~\ref{lem1}). 

If $Y(x)$ is a formal fundamental solution of $x\frac{d Y(x)}{d x} = A(x)Y(x)$, then as seen above $Z(x)=B(x)Y(x)$ is a formal fundamental 
solution of $x\frac{dZ(x)}{dx} = A(qx)Z(x)$. Comparing it with the formal fundamental solution
$\tilde Z(x)=Y(qx)$ of this equation, it follows that $0$ is a regular singular point (see Lemma~\ref{lem1a}). A similar statement holds for $\infty$. We then have that if $Y(x)$ is a fundamental solution analytic in a neighborhood of an ordinary point, then $Y(x)$ can be analytically continued to a meromorphic function on the universal cover $\hat{\CX}$ of $\CX\backslash\{0\}$. If $Y(xe^{2\pi i})$ is the solution matrix obtained by analytically continuing $Y(x)$ once around $0$, we have that  $Y(xe^{2\pi i}) = Y(x) H$ for some $H \in \GL_n(\CX)$. Writing $H = e^{2\pi i \tilde{A}}$  with a non-resonant $\tilde A$, a calculation shows that $G(x)=Y(x)x^{-\tilde{A}}$ is a matrix valued meromorphic function on $\CX\backslash\{0\}$. Using the fact that $0$ and $\infty$ are regular singular points, one deduces that $G(x)$ has moderate growth at these points and so must have rational function entries.  Therefore  the gauge transformation $y(x) = G(x) z(x)$ transforms $x\frac{d y(x)}{d x} = A(x)y(x)$ to $x\frac{d z(x)}{d x} = \tilde{A}z(x)$.  One then shows that the consistency condition implies that this transformation also results in a constant $q$-difference equation for $z(x)$ and so (\ref{eq000}) is transformed into (\ref{eq001}) (see Lemma~\ref{lemQM} for details).}

The rest of our work is organized as follows.  In Section~\ref{Sec2} we consider systems (\ref{eq000}) where $\sigma(x) = x+1, qx \ (q\neq 0$
not a root of unity)  or $ x^q$, ($q$ an integer $\geq 2$) and $A(x) \in \gl_n(C(x)),  B(x) \in \GL_n(C(x)),\ C$ an algebraically closed field of characteristic zero. Assuming a consistency condition analogous to (\ref{eq000a}) we show in Theorem~\ref{thm1} that there is a transformation $Y(x) = D(x)Z(x)$ with $D(x) \in GL_n(C(x))$ taking system (\ref{eq000}) to a much simpler system.  
When $\sigma(x) = qx$ or $x^q$, we characterize those $y(x) \in C[[x]][x^{-1}]$  and those $y(x)$ meromorphic on the Riemann surface of $\log x$ (when $C = \CX$) that simultaneously satisfy a linear differential equation and a linear $\sigma$-difference equation over $C(x)$ ( {Corollary}~\ref{cor0}). When $\sigma(x) = x+1$, we  characterize those $y(x) \in C[[x^{-1}]][x]$  and those $y(x)$ meromorphic on $\CX$ that simultaneously satisfy a linear differential equation and a linear $\sigma$-difference equation over $C(x)$ (Corollary~\ref{cor0a}). Theorem~\ref{thm1} allows us to also characterize in Corollary~\ref{cortime} when the time-$1$-operator of a linear differential system has rational entries.
  
  In Section~\ref{Sec2a} we consider systems of the form
  \begin{equation}\label{002}
\begin{aligned}
\sigma_j(Y) & = & B_j\,Y,\ j=1,2
\end{aligned} 
\end{equation}
with $B_j \in \GL_n(C(x))$ satisfying a suitable consistency condition and $(\sigma_1,\sigma_2)$ defined by $(\sigma_1(x) = x+1, \sigma_2(x) = x+\alpha), \alpha \in \CX\backslash \QX$ or $(\sigma_1(x) = q_1x, \sigma_2(x) = q_2x)$ or $(\sigma_1(x) = x^{q_1}, \sigma_2(x) = x^{q_2})$ with $q_1$ and $q_2$ multiplicatively independent.  Theorem~\ref{thm2-1} states that  there is a gauge transformation $Y(x)= D(x) Z(x), D(x) \in \GL_n(C(x))$ in the first two cases and $D(x) \in \GL_n(K), K=C(\{x^{1/s}\mid s\in\NX^*\})$ in the last case, transforming such a system into  {a system}  with constant  {coefficients}.  Once again the proofs depend on showing that the singular points  and the connection relations  are particularly simple.  We again have corollaries characterizing formal solutions and solutions on various domains  of two linear $\sigma$ equations in each of these three cases (Corollaries~\ref{cor2s},~\ref{cor2q} and~\ref{cor2m}).
 
 We end this introduction with a discussion of two applications of our results.  The first concerns properties of {\it automatic sets} (See \cite{AS2003} for a general introduction to these sets and \cite{Becker94} and \cite{Ran92}  for connections to Mahler equations). 
A subset $\calN \subset \NX$ of integers is called {\it k-automatic} if there is a finite-state machine   that accepts as input the base-$k$ representation of an integer and outputs $1$ if the integer is in $\calN$ and $0$ if it is not in $\calN$. Many  {sets} can be $k$-automatic for fixed $k$ (for example the set of powers of $2$ is $2$-automatic) but only very simple sets can be $k$- and $\ell$-automatic for multiplicatively independent integers $k$ and $\ell$.  This fact is formalized in Cobham's Theorem \cite{Cobham}, \cite{Durand}. \\
 
\noindent {\bf Theorem (Cobham).} {\it Let $k$ and $\ell$ be two multiplicatively independent integers. Then a set $\calN \subset \NX$ is both $k$- and $\ell$-automatic if and only if it is the union of a finite set and a finite number of arithmetic progressions.}\\

Linear difference equations involving the Mahler operator and $k$-automatic sets are related by the following fact:  If $\calN$ is a $k$-automatic set then $F(x) = \sum_{n\in \calN} x^n$ satisfies a scalar linear difference equation over $\QX(x)$ with respect to the Mahler operator $\sigma(x) = x^k$, that is, a $k$-Mahler equation. In (\cite{AuB16}, {Theorem 1.1}), Adamczewski and Bell show: a power series $f(x) \in C[[x]][x^{-1}]$ satisfies both a $k$- and $\ell$-Mahler equation if and only if it is a rational function, proving a conjecture of Loxton and van der Poorten \cite{Poorten}.  Their proof relies on Cobham's Theorem.  {On the other hand, it is known that their Theorem 1.1 implies Cobham's Theorem (see, for example, Section 2, \cite{AuB16} or Chapitre 7, \cite{Ran92}).}  In our work we prove and generalize the Adamczewski-Bell result (Corollary~\ref{cor2m}) without using Cobham's Theorem, therefore yielding a new proof of this latter result. {Our proof of  the Adamczewski-Bell result follows the general philosophy of our work.  We show that proving that a power series  of two such Mahler equations is rational can be reduced to showing that consistent pairs of first order Mahler systems must be of a very simple nature. In fact, although we deduce the Adamczewski-Bell result from Theorem~\ref{thm2-1} mentioned above, we do not need its full strength and can also prove this result from the weaker statement contained in Proposition~\ref{2mprop}.}
 
 The second application concerns the Galois theory of difference equations. In \cite{HaSi08} a differential Galois theory of linear difference equations   was developed as a tool to understand the differential properties of solutions of linear difference equations.  This theory associates to a system of linear difference equations $Y(\sigma(x)) = B(x) Y(x)$  a group called the {\it differential Galois group}. This  is a linear differential algebraic group, that is a group of matrices whose entries are functions satisfying a fixed set of (not necessarily linear) differential equations.  Differential properties of solutions of the linear difference equation are measured by group theoretic properties of the associated group.  For example, a group theoretic proof is given in \cite{HaSi08}  of H\"older's Theorem that the Gamma Function satisfies no polynomial differential equation, that is, the Gamma Function is hypertranscendental.   In general, one can measure the amount of differential dependence among the entries of a fundamental solution matrix of $Y(\sigma(x)) = B(x) Y(x)$  by the size of its associated group; the larger the group the fewer differential relations hold among these entries.  This theme has been taken up in \cite{DHR15} where the authors develop criteria to show that the generating series $F(x) = \sum_{n\in \calN} x^n$ of certain $k$-automatic sets $\calN$ are hypertranscendental.  As we mentioned above, these generating series satisfy Mahler equations and Dreyfus, Hardouin and Roques in  \cite{DHR15} develop criteria to insure that a given Mahler equation has $\SL_n$ or $\GL_n$ as its associated group.  The proofs of the validity of their criteria depend on B\'ezivin's result \cite{Bez94} that a power series that simultaneously satisfies a Mahler equation and a linear differential equation must be a rational function.   In \cite{DHR16}, the authors develop similar criteria (using the result of Ramis mentioned in the Introduction) for linear $q$-difference equations.
 Both Ramis's result and B\'ezivin's result  appear in Corollary~\ref{cor0} as a consequence of Theorem~\ref{thm1} in the present work.  Using Theorem~\ref{thm1} directly, the authors of \cite{AS16} classify the differential Galois groups that can occur for the equations considered in this latter theorem and, in particular, rule out certain groups from occurring.  Using this classification, it is shown in \cite{AS16} how the criteria of \cite{DHR15} and  \cite{DHR16} can be extended and given simple proofs. The results of \cite{AS16} can also be used in designing algorithms to compute  the differential Galois group of linear difference equations (c.f.,  \cite{Arreche15b}). {Using the Galois theory 
presented  in~\cite{OW15}, the authors of~\cite{DHR16} also develop criteria  to determine when a solution of a linear $q$-difference equation satisfies no $q'$-difference relation (even nonlinear) with respect to a multiplicatively independent $q'$. This is done, in a manner analogous to the results of~\cite{DHR15}, by developing criteria to insure that the Galois groups in this context are large. 
 Their result depends on the results of B\'ezivin and Boutabaa~\cite{BB92}. The results of Section~\ref{Sec2a} can also be used to sharpen the criteria in~\cite{DHR16}.
 } 

\section{Reduction of systems of differential and difference equations}\label{Sec2} 

Let $C$ be an algebraically closed field\footnote{All fields considered in this work are of characteristic zero.} and $k = C(x)$. Let $\delta$ be a derivation on $k$ with constants $C$ and $\sigma$ be a $C$-algebra 
endomorphism of $k$. 
We suppose that there is a  constant $\mu\in C$ such that $\delta\sigma=\mu\,\sigma\delta$. {This commutativity
except for a constant factor is crucial for our approach.}

We consider three cases of couples $(\delta,\sigma)$ below.
\begin{itemize}
\item[{\rm case S:}] The derivation is $\delta=d/dx$ and $\sigma$ is the shift operator defined by
$\sigma(x)=x+1$. Here $\mu=1$.
\item[{\rm case Q:}] The derivation is $\delta=x\,d/dx$ and $\sigma$ is the $q$-dilation operator 
defined by $\sigma(x)=q\,x$ with some $q\in C$, $q\neq0$ and not a root of unity. Note that  $\mu=1$ here as well.
\item[{\rm case M:}] The derivation is $\delta=x\,d/dx$ and $\sigma$ is the Mahler operator 
defined by $\sigma(x)=x^q$ with an integer $q\geq2$. Here we have $\mu=q$.
\end{itemize}
Observe that $\sigma$ is bijective in  cases S and Q, but not in  case M.

We will consider systems 
\begin{equation}\label{eq1}
\begin{aligned}
\delta(Y) & = & AY\\
\sigma(Y) & = & BY
\end{aligned} 
\end{equation}
with $A \in \gl_n(k), B \in \GL_n(k)$ that are {\it consistent}, that is $A$ and $B$ satisfy the  consistency condition  given by
\begin{equation} \label{eq2}
\begin{aligned}
\delta(B) &=& \mu\,\sigma(A)B-BA.
\end{aligned}
\end{equation}

{{The consistency condition is closely related to the almost-commutativity of $\delta,\,\sigma$.} Note that{ it} guarantees that $\delta(\sigma(Z)) = \mu\sigma(\delta(Z))$ holds for any solution $Z$ of the system { (\ref{eq1})} in any extension of $C(x)$. It is satisfied if there exists
a fundamental solution of $\delta(Y) =  AY$ that is also a solution of $\sigma(Y) = BY$ in some
extension of $C(x)$ { in which $\delta$ and $\sigma$ commute} or if there exists a solution vector of the system in 
{ such an}  extension
such that its components are linearly independent over $C(x)$. The consistency condition
is satisfied for the systems (\ref{eq1}) constructed from the applications to common solutions 
of pairs of linear scalar equations, {again because $\delta$ and $\sigma$ commute except for a constant factor.}}

We say that (\ref{eq1}) is {\it equivalent} (over $k$) to a system 
\begin{equation}\label{eq3}
\begin{aligned}
\delta(Z) & = & \tilde{A}Z\\
\sigma(Z) & = & \tilde{B}Z
\end{aligned} 
\end{equation}
with $\tilde A \in \gl_n(k), \tilde B \in \GL_n(k)$ if for some $G\in \GL_n(k)$,
\begin{equation}\label{eq4}
\begin{aligned}
\tilde{A} & =  \delta(G)G^{-1} + GAG^{-1}\\
\tilde{B} & =  \sigma(G)BG^{-1}
\end{aligned} 
\end{equation}
that is, if  (\ref{eq3}) comes from (\ref{eq1}) via the {\it gauge transformation} $Z = GY$. Note that the property of consistency is preserved under equivalence. 

A simple, but crucial observation is the fact that the consistency condition can be expressed
as an equivalence.
\begin{lemma}\label{keyobs} Consider the system (\ref{eq1}). 
It satisfies the consistency condition (\ref{eq2}) if and only if
it is equivalent to the system
\begin{equation}\label{conj-sys}\delta(Z)=\mu\sigma(A)Z,\ \ \sigma(Z)=\sigma(B)Z\end{equation}
by the gauge transformation $Z=BY$.
For $N\in\NX^*$, it is equivalent to the systems
\begin{equation}\label{conj-sysN}\delta(Z)=\mu^N\sigma^N(A)Z,\ \ \sigma(Z)=\sigma^N(B)Z.\end{equation}
\end{lemma}
\begin{proof} Rewriting (\ref{eq2}) yields the first part, iteration using the fact that all
these equivalent systems are again consistent yields the second.\end{proof}

\noindent Observe that (\ref{conj-sysN}) is also obtained by applying $\sigma^N$ to (\ref{eq1}). 

The main result of this section { expresses that the consistency condition is very restrictive.}
\begin{theorem}\label{thm1} The system  (\ref{eq1}) satisfying the consistency condition (\ref{eq2})  
is equivalent over $k$  to a system (\ref{eq3})
with $\tilde{A}\in \gl_n(C)$, $\tilde B\in\GL_n(k)$. Moreover:
\begin{itemize}
\item[{\rm case S:}] $\tilde A$ is diagonal, $\tilde B\in \GL_n(C)$ is constant and upper triangular and commutes with $\tilde A$.
\item[{\rm case Q:}] If $\lambda_1,\lambda_2$ are eigenvalues of $\tilde A$,  then $\lambda_1-\lambda_2\not\in\ZX\setminus\{0\}$.
 $\tilde B\in \GL_n(C)$ is constant and commutes with $\tilde A$.
\item [{\rm case M:}] The eigenvalues of $\tilde A$ are rational and in the interval $[0,1[$ and 
there exists a diagonalisable matrix $D$ with integer eigenvalues
commuting with $\tilde A$ such that $\tilde A+D$ is conjugate to $q\,\tilde A$.
We have
$\tilde{B} \in \GL_n(C[x,x^{-1}])$,  such that 
 the exponents $m$ appearing with nonzero coefficient in $\tilde B$
are integer differences of the form $q\lambda_1-\lambda_2$ of eigenvalues $\lambda_1,\lambda_2$ of $\tilde A$. 
\end{itemize}
\end{theorem}
{Remark: Simple counter-examples show that the statement in case Q no longer holds if $q$
is a root of unity. Consider for instance the natural consistent system satisfied by
$y_j(x)=\exp(q^jx)$, $j=0,...,n-1$, if $q^n=1$.}

Before giving a proof of this result {in sections~\ref{th1sq} and \ref{thm1m}}, we deduce 
several corollaries {concerning common solutions of the linear differential and $\sigma$-difference equations
\begin{equation}\label{sysLS}\renewcommand{\arraystretch}{1.4}
\begin{array}{rcl}L(f(x)) &= & \delta^n(f(x)) + a_{n-1}(x) \delta^{n-1}(f(x)) + \ldots + a_0(x) f(x) =  0 
\mbox { and }\\
S(f(x)) & =  &\sigma^m(f(x)) + b_{m-1}(x) \sigma^{m-1}(f(x)) + \ldots + b_0(x)f(x) =  0,\end{array}\end{equation}
with $a_i(x), b_i(x) \in \CX(x)$.}
\begin{cor}\label{cor0} Consider $\delta,\sigma$ as in case Q or M. Let $E$ be the field of meromorphic functions on $\hat\CX$, where $\hat\CX$ denotes the Riemann surface of the logarithm over $\CX\setminus\{0\}$. $\CX(x)$ is considered as a subfield of $E$. If $f\in E$ satisfies the linear differential and $\sigma$-difference equations (\ref{sysLS})
then 
\begin{equation}\label{ratilog}{ f(x)  = \sum_{i,j=1}^t r_{ij}(x) x^{\alpha_i}\log(x)^j}\end{equation}
where $\alpha_i \in \CX$ and $r_{ij}\in\CX(x)$. In case M, we obtain moreover that {the} $\alpha_i$ are rational.
Conversely, any such function satisfies a pair of linear differential and $\sigma$-difference equations with coefficients in $\CX(x)$.

{Assume $f(x) \in \CX[[x]][x^{-1}]$ with $\CX(x)$ considered as a subfield of $\CX[[x]][x^{-1}]$. If $f(x)$ satisfies the linear differential and $\sigma$-difference equations
(\ref{sysLS}) then $f$ is rational, \ie\ $f\in\CX(x)$.     }
\end{cor}
\noindent{\bf Remark:} 1.\ { The Corollary can be extended to functions satisfying 
non-homogeneous systems 
$$L(f(x))=h_1(x),\ \ S(f(x))=h_2(x),$$
where $h_1(x)$ and $h_2(x)$ are of the form (\ref{ratilog}). Indeed, as the latter satisfy homogeneous
systems of the form (\ref{LS}), it is straightforward to eliminate them from the non-homogeneous 
system at the expense of increasing the orders $n$ and $m$ of the equations. A similar remark applies 
to the corresponding corollaries in case S and in cases 2S and 2Q in section \ref{Sec2a}. 
}\\[0.1in]
2.\ { Functions of the form (\ref{ratilog})} are elementary function. Algorithms for finding such solutions of linear difference equations are known (eg. \cite{DaSi}, \cite{PuSi2003}, Ch.~4) and can be modified to find solutions of this special type. Possible values of  $\alpha$ and relevant powers of $\log x$ can be calculated using effective procedures to determine canonical forms of such equations at singular points (\cite{PuSi2003}, Ch.~3.1). { Again a} similar remark applies to the corresponding corollaries
in  case S and in  cases 2S and 2Q of section \ref{Sec2a} (\cite{PuSi}).\\[0.1in]
{3.\ The final assertion of the above corollary corresponds to  results of Ramis \cite{Ramis92} in case Q and of B\'ezivin \cite{Bez94} in case M. Their proofs proceed by examining the asymptotic behavior of solutions of the two scalar linear differential equations rather than our approach.
{In \cite{Ramis92}, Ramis assumed that $q\neq0$ and $q$ is {transcendental} if $|q|=1$. This condition can be reduced to $|q|\neq0,1$.
The latter was needed to ensure that $q^nx\to 0$ or $q^nx\to\infty$ when $n\to\infty$
and allowed asymptotic results for $q$-difference equations to be applied.
As mentioned above, our approach only requires that
$\sigma(x) = qx, q \neq0$ has no periodic points other than the fixed points.}
}\\[0.1in]
{4.\ Although we use  Theorem~\ref{thm1} to prove the final statement of this corollary, one can prove this directly from Lemmas~\ref{lem1} and~\ref{lem1a} below as noted in the remark following Lemma~\ref{lem1a}.}

\begin{prf} {We begin by proving the assertion concerning a function $f(x) \in E$. }Let $W_0$ be the $\CX(x)$-subspace of  $E$ spanned by 
\[\{\sigma^j\delta^i(f(x))\}\]
where $0\leq i \leq n-1$ and $0 \leq j \leq m-1$. Using the fact that $\delta\sigma=\mu\sigma\delta$ and the equation $L(f(x)) = 0$, one sees that $W_0$ is left invariant under $\delta$.  Similarly, using $S(f(x)) = 0$,  one sees that $W_0$ is invariant under $\sigma$. Unfortunately, $\sigma$ does not preserve linear independence over $\CX(x)$ in case M -- just consider $1$ and $x^{1/q}$.
Therefore we consider now the vector spaces $W_\ell$ generated by the elements of $\sigma^\ell(W_0)$, $\ell=0,1,...$. These are again invariant
under $\delta$ and $\sigma$. As they form a
descending chain of finite dimensional $\CX(x)$-vector spaces there must be a first index $s$ such $W_s=W_{s+1}$. It is easy to see that we then have that  $W_\ell=W_{\ell+1}$ for all $\ell\geq s$.
In the sequel, we consider $W_s$ and omit the index $s$.
 
Let $w_1, \ldots , w_t$ be a $\CX(x)$-basis of $W$ and let $w = (w_1, \ldots , w_t)^T$.  We  have that
 \begin{equation}\label{eq5}
\begin{aligned}
\delta(w) & = & Aw\\
\sigma(w) & = & Bw
\end{aligned} 
\end{equation}
for some $A\in \gl_t(k), B\in \GL_t(k)$ because $\sigma(w_1),...,\sigma(w_t)$ again generate $W$.  
We claim that $A$ and $B$ satisfy (\ref{eq2}).  To see this,  note that 
 \[0=\mu\sigma\delta(w) - \delta\sigma(w) =(\mu\sigma (A)B-BA-\delta(B))w.\]
Since the entries of $w$ are linearly independent over $k$, we have (\ref{eq2}).

{Before we continue with the proof of this corollary, we remark that the above argument} does not use special properties of $\delta$ and $\sigma$ and will again be useful later.
We note it as
\begin{lemma}\label{LS} Let $C$  be an algebraically closed field, $k=C(x)$, $E$ a $k$-algebra, $\delta$ a derivation on $E$ annihilating $C$ satisfying $\delta(x)\in k$
and $\sigma$ a $C$-algebra endomorphism on $E$ satisfying $\sigma(k)\subset k$ and $\delta\sigma=\mu\sigma\delta$ for some
$\mu\in C^*$. Suppose that there exists an $f\in E$ satisfying a system (\ref{sysLS}) of equations. Then there exist $s,t\in\NX$ \footnote{$\NX$ denotes the set of non-negative integers in the present text.},
a solution vector { $w=(w_1,...,w_t)^T\in E^t$} of a system (\ref{eq5}) satisfying the consistency condition (\ref{eq2}) and $r_i\in k$, $i=1,..,t$,
such that $\sigma^s f=\sum_{i=1}^t r_i w_i$.
\end{lemma}

We now apply Theorem~\ref{thm1} to the equations $\delta(Y) =AY, \sigma(Y) = BY$.  We conclude that there is a gauge transformation $Z=GY$ that transforms this system to a system $\delta(Z)  =  \tilde{A}Z,
\sigma(Z)  =  \tilde{B}(x)Z$ where $\tilde A$ is a constant matrix (with rational eigenvalues in case $M$).  Letting $z=(z_1,\ldots , z_t)^T = Gw$ we see that 
$z\in E^t$ satisfies $\delta z= \tilde A z$ and hence $z=x^{\tilde A}C$ for some constant matrix $C$.
Therefore  the $z_i$ are of the desired form. 
 Since the $z_i$ are again a $\CX(x)$-basis of $W$, there exists a non-negative integer $s$ such that $\sigma^s(f)$ is a $\CX(x)$-linear
combination of the  $z_i$. In case Q, it follows immediately that  $f$ is also a $\CX(x)$-linear combination of 
terms of the wanted form. In  case $M$, we find that $f$ is a $\CX(x^{q^{-s}})$-linear combination of terms of the desired form, that is, it is itself of that form. Thus we have proved  the first statement of the corollary.

To prove the second part of the corollary, first assume that $f=x^{\alpha}\log(x)^j$ with $\alpha\in\CX$, $j\in\NX$.  A simple calculation shows that the operator $L = (\delta -\alpha)^{j+1}$ annihilates $f$.

In case Q, we obtain that $S = (\sigma-q^{\alpha})^{j+1}$ also annihilates $f$.

In case M, it is required that $\alpha$ be rational. As there are only finitely many rationals in $[0,1[$ 
with the same denominator as $\alpha$, there must exist
positive integers $m\neq n$ such that $(q^m-q^n)\alpha=:r\in\ZX$. Then $\sigma^m(f)=q^{j(m-n)}x^r\sigma^n(f)$ and we have found a $\sigma$-difference equation for $f$.

The general case follows from the fact that sums and products with rational functions of solutions of differential 
{or $\sigma$-difference equations, respectively, } again satisfy such equations.

{We now turn to the assertion concerning a function $f(x) \in \CX[[x]][x^{-1}]$. We can again use Lemma \ref{LS} and apply Theorem \ref{thm1}.
{ Following the proof of the first part of  this corollary }{we will }find the solutions of
$\delta z=\tilde Az$ in $\CX[[x]][x^{-1}]^t$. Clearly the coefficient $z^{(\ell)}$ of $z$ in front of $x^\ell$ must satisfy
$\tilde A\,z^{(\ell)}=\ell\,z^{(\ell)}$. Hence $z^{(\ell)}$ can only be different from 0 for finitely many $\ell$ and so {$z\in\CX[x,x^{-1}]^t\subset\CX(x)^t$}. Since the entries in the vector $z$ form a $\CX(x)$-basis of $W$, we must have $t=1$. Therefore $z$ is a scalar in $\CX(x)$ and so  $\sigma^s(f)$ is also rational, $\sigma^s(f)(x)=g(x)\in\CX(x)$.  In case Q, the proof is finished since we can immediately conclude that $f(x)\in\CX(x)$.
In case M, we have that
$\sigma^s(f)$ is in $\CX\{x^{q^s}\}[x^{-q^s}]$ and hence invariant under 
the substitutions $x\to x\exp(2 \pi i \ell/q^s)$, $\ell=0,...,q^s-1$. Therefore we also have that
$$\sigma^s(f)(x)=q^{-s}\sum_{\ell=0}^{q^s-1}g(x\exp(2 \pi i \ell/q^s) )\in\CX(x^{q^s}).$$ 
Hence $f\in\CX(x)$.}
\end{prf}

{ In the shift case, the corollaries obtained are slightly different as we cannot consider the difference equation on 
$\hat \CX$ or $\CX[[x]][x^{-1}]$. It is possible here to consider the difference equation for meromorphic functions of $\CX$
and for formal Laurent series in $1/x$, that is in $\CX[[x^{-1}]][x]$. Again $\CX(x)$ is considered as a subset of
$\CX[[x^{-1}]][x]$, where the inclusion is given by series expansion at $x=\infty$.}
\begin{cor}\label{cor0a} Consider $\delta=d/dx$ and $\sigma$ induced by $\sigma(x)=x+1$ . 
If $f(x)$ is a meromorphic function {on some horizontal strip $\{x\in\CX\mid m<\Im x<M\}$}
satisfying the system (\ref{sysLS}) of a linear differential and difference equations
with coefficients in $\CX(x)$ then 
\[ f(x)  = \sum_{i=1}^t r_i(x) e^{\alpha_i x} \]
where $\alpha_i \in \CX$ and $r_i(x) \in \CX(x)$. Conversely, any such function satisfies a pair of linear differential and difference equations with rational coefficients.

{ If $f(x)\in \CX[[x^{-1}]][x]$  satisfies the system (\ref{sysLS}) of a linear differential and difference equations
with coefficients in $\CX(x)$ then $f(x)$ is rational. }
\end{cor}

{\noindent{\bf Remark:} In \cite{BG96}, B\'ezivin and Gramain consider {\it entire} solutions of linear differential/dif\-ference equations of the type considered in Corollary~\ref{cor0a} and show that such solutions must be of the form described. Their proof ultimately depends on an analytic result of Kelleher and Taylor describing the growth properties of a set of entire functions equivalent to the property  that they generate an ideal equal to the whole ring of entire functions.   A crucial part of their proof also involves reduction to equations with constant coefficients but in a different manner than the proof presented here. They also give conditions on the equations that guarantee that the $r_i(x)$ appearing in the expression for $f(x)$ are polynomials.

Such a condition can also be given here. Let $\ell(x)$ be the least common denominator of the coefficients
$a_j(x)$, $j=0,...,n-1$ of $L$ in (\ref{sysLS}), $r(x)$ the least common denominator of the coefficients $b_j(x)$, $j=0,...,m-1$
of $S$.  By the $\CX(x)$-linear independence of $e^{\alpha_i x}$, $i=1,...,t$, 
also the individual terms $r_i(x) e^{\alpha_i x}$ are solutions of the scalar equations (\ref{sysLS}). Therefore
the poles of some $r_i(x)$ must be among the zeroes of $\ell(x)$.
Now let $\alpha$ be a pole of some $r_i(x)$ with minimal real part. Then it is a zero of $\ell(x)$ as seen above,
but in view of {$\sigma^{-m}S(r_i(x) e^{\alpha_i x})=0$}, $\alpha$ must also be a zero of $\sigma^{-m}r$.
As a consequence, all $r_i$ must be polynomials, if $\ell$ and  $\sigma^{-m}r$ have no common zero.}

\begin{prf} Using Lemma \ref{LS} and Theorem \ref{thm1}, we obtain that 
$f(x)$ is a $\CX(x)$-linear combination
of solutions of    $\delta z=\tilde{A}z$ with diagonal 
constant $\tilde{A}$. This yields the first part of the corollary.

To prove the second part of the corollary, first assume that the $r_i$ are polynomials with  $\deg(p_i) = n_i$.  A simple calculation shows that the operators $L = \prod_{i=1}^t(\delta -\alpha_i)^{n_i+1}$ and $S = \prod_{i=1}^t (\sigma-e^{\alpha_i})^{n_i+1}$ annihilate $f$. 
In general, let $p(x)$ be a common denominator of the $r_i$ and let $\tilde L$ and $\tilde S$ be the operators annihilating $p(x)f$.
Then clearly $L=\frac1{p(x)}\tilde L\circ M_{p(x)}$ and $S=\frac1{\sigma^n p(x)}\tilde S\circ M_{p(x)}$, $n=t+\sum n_i$, $M_{p(x)}$ the multiplication operator, satisfy the conditions of the theorem.

{ For the last part of the corollary we use again  Lemma \ref{LS} and Theorem \ref{thm1} and obtain again that 
$f(x)$ is a $\CX(x)$-linear combination
of solutions of    $\delta z=\tilde{A}z$, with diagonal 
constant $\tilde{A}$, now in $\CX[[x^{-1}]][x]^t$. 
Hence $z$ is also constant  
\footnote{By the way, as the entries of $z$ form a $\CX(x)$-basis of the space $W$ of the proof of Lemma \ref{LS}, 
we have $t=1$.}. As a consequence, $f(x)$ is in $\CX(x)$.}
\end{prf}

Another consequence of Theorem \ref{thm1} in case S concerns the {\it time-1-operator} 
associated to a system $\frac{dy}{dx}=A(x)y$, $A\in\CX(x)$.
It is defined by $T_A(x)=Y(x+1,x)$, where $Y(t,x)$ denotes the solution matrix of $\frac{d}{dt}Y(t,x)=A(t)Y(t,x)$ satisfying $Y(x,x)=I$.
It is holomorphic on $\CX\setminus\cup_{\ell=1}^L[x_\ell-1,x_\ell+1]$ if $x_\ell$, $\ell=1,...,L$ are the singularities of $A(x)$. By the theorem on the dependence of solutions upon initial conditions, it is readily verified that
$W(t,x)=\frac{\partial Y}{\partial x}(t,x)$ satisfies $\frac{d}{dt}W(t,x)=A(t)W(t,x)$ and $W(x,x)=-A(x)$ and 
hence $\frac{\partial Y}{\partial x}(t,x)=-Y(t,x)A(x)$. Thus we find that
$$\frac d{dx} T_A(x)=A(x+1)T_A(x)-T_A(x)A(x).$$
The system
\begin{equation}\label{time1}\begin{array}{rcl}
\frac d{dx} y &=&A(x)y\\
y(x+1)&=&T_A(x)y(x)\end{array}
\end{equation}
therefore satisfies the consistency condition (\ref{eq2}). 
This implies that $T_A(x)$ may be continued analytically to the universal covering of $\CX\setminus\{\xi,\xi+1\mid \xi \mbox{ a singularity of }A(x)\}$.
It also allows us to show
\begin{cor}\label{cortime} 
The time-1-operator defined for a system of linear differential equations $\frac{dy}{dx}=A(x)y$, $A(x)\in\gl_n(\CX(x))$, 
has rational functions as entries if and only if the system is equivalent over $\CX(x)$ to a system with a constant
diagonal coefficient matrix.\end{cor}
\begin{prf}It suffices to put $B(x)=T_A(x)$ and to apply Theorem 1.\end{prf}
An analogous corollary can be stated for the ``$q$-multiplication-operator'' that can be defined for a system 
$x\frac{dy}{dx}y=A(x)y$. Details are left to the reader.

In  case M, the statement of Theorem 1 is not entirely satisfactory, because the matrix $B(x)$ obtained is 
not constant. This can be repaired by changing the base field. Consider the field 
$K=C(\{x^{1/\ell}\mid \ell\in\NX\})$ (see \cite{DHR15}) containing all fractional powers of $x$. This field also has the advantage 
that $\sigma$, now defined by mapping $x^\alpha$ to $x^{q\alpha}$ for all rational $\alpha$, is an automorphism\footnote{ For this property it would not be necessary to have all fractional powers of $x$ in the field. The field $K_0 = C(\{x^{1/q^\ell}\mid \ell\in\NX\})$ would suffice.}. 
{ The derivation  $\tilde\delta=x\log(x)\frac {d}{dx}$ and the automorphism  $\sigma$ mapping $x^\alpha$ to $x^{q\alpha}$,
$\log(x)$ to $q\log(x)$  commute on the base field $K(\log x)$ ({\it i.e.} $\mu = 1$). Furthermore,  as} the consistency conditions for the couple $(\delta,\sigma)$ and the couple
$(\tilde\delta,\sigma)$ are equivalent we do not consider $\tilde\delta$ here.
\begin{cor}\label{corMK} Consider the system  (\ref{eq1}) in case M with $A\in\gl_n(K),B\in\GL_n(K)$ satisfying the consistency condition (\ref{eq2}).
It is equivalent over $K$  to a system (\ref{eq3})
with nilpotent constant $\tilde{A}\in \gl_n(C)$ and  constant $\tilde B\in\GL_n(C)$ satisfying
$q\tilde A\tilde B=\tilde B \tilde A$.
It is equivalent over $K(\log(x))$  to a system (\ref{eq3})
with $\tilde{A}=0$ and  constant $\tilde B\in\GL_n(C)$.
\end{cor}
\begin{prf} Making a change of variables $x=t^N$ with a suitable positive integer $N$, if necessary, we can assume that $A,B$ 
have entries in $C(x)$. By Theorem 1, there is a gauge transformation changing (\ref{eq1}) into $\delta U=\bar A U,\,\sigma U=\bar B(x) U$  
with constant $\bar A$ having rational eigenvalues. By a constant transformation, we can assume {that} $\bar A$ is in Jordan canonical form.
If $D=\diag(r_1,...,r_n)$ is the diagonal of $\bar A$, then $U=\diag(x^{r_1},...,x^{r_n}) Z$ is a gauge transformation in
 $\GL_n(K)$ changing the system into (\ref{eq3}) where now $\tilde A=\bar A-D$ is nilpotent. The matrix $\tilde B(x)$ has entries in
$C[x^{1/N},x^{-1/N}]$ for some positive integer  $N$ and satisfies the consistency condition. Writing
$\tilde B(x)=\sum_{m={ -t}}^{t} B_m x^{m/N}$ yields $q\tilde A\,B_m-B_m\tilde A=\frac mN B_m$ for any $m\in\{-t,\ldots ,t\}$.
Since ${\tilde A}$ is nilpotent, so is the operator mapping $X$ to $q\tilde A\,X-X\tilde A$. Hence $B_m=0$ unless $m=0$.
This proves the first statement.

In order to prove the second statement it suffices to use the first and to make the gauge transformation $V=\exp(-\log(x) \tilde A) Z$.
Observe that the consistency condition for $\tilde A,\tilde B$ is $q\tilde A\tilde B=\tilde B\tilde A$ and that is implies
$\exp(q\log(x) \tilde A)\tilde B=\tilde B \exp(\log(x) \tilde A)$.\end{prf}

We now turn to a proof of Theorem~\ref{thm1}.  
{{In} a first step, the consistency condition will be used in the form of Lemma \ref{keyobs}} to characterize the singularities of the equation $\delta(Y) = A(x)Y$.  {We say that  a singular point $x_1$ of the equation $\delta(Y) = A(x)Y$ is an {\it apparent singular point} if there is a fundamental solution matrix  whose entries are in 
{$C\{x-x_1\}[(x-x_1)^{-1}]$}
\footnote{Although some authors use this term to mean that the equation has a  fundamental solution matrix {\it holomorphic} at $x_1$, we will use the above extended meaning throughout this work.}.  Note that this condition is the same {as} saying that there is an equivalent system for which $x_1$ is a regular point.  To see this note that truncating the entries of such a fundamental solution matrix at a sufficiently high power, we obtain a matrix $G$ such the gauge transformation $Y=GZ$ leads to an equivalent system for which $x_1$ is a regular point.}  In Lemma~\ref{lem1} we show that the finite singular points must be apparent  singular points 
(with the exception of 0 in  cases Q and M).
In  cases Q and M, Lemma~\ref{lem1a} shows the effect of 
{consistency, that is Lemma \ref{keyobs}}, on 
the structure of local solutions at these points.  
Finally, monodromy theory is used to show that the differential 
equation is equivalent to $\delta z=Mz$  with some constant matrix $M$ 
and that the spectrum of $M$ has the desired properties.
In case S,  Lemma~\ref{lem1b} states the first consequence for
the solutions at infinity. Then we have to work considerably harder to arrive at 
the wanted reduced form of the differential equation (c.f., Proposition \ref{prop1}.) 
The rest of the conclusions of Theorem~\ref{thm1} will follow easily.

The entries of $A$ and $B$ have coefficients that lie in a countable algebraically closed field so we may replace $C$, 
if necessary, with a countable algebraically closed field, again denoted by $C$. 
Furthermore, we may assume that $C$ is a subfield of the complex numbers $\CX$. 
In the rest of the proof it can be verified that all equivalent systems can be chosen to have entries in this $C(x)$.

\begin{lemma}\label{lem1} Consider $\delta,\sigma$ as in  cases S,Q or M 
 and a system (\ref{eq1}) satisfying the consistency condition 
(\ref{eq2}). Then each finite singular point $\xi$ of the differential 
equation $\delta Y=A\,Y$, except maybe 0 in  cases Q and M, 
is  an apparent singular point. 
\end{lemma}

\noindent{\bf Remark:}  Using the Lemma, it can be shown (see e.g.\ \cite{BaMa}) that 
the system (\ref{eq1}) is equivalent to a 
system (\ref{eq3}) where $\tilde{A} \mbox{ and } \tilde{B}$ have entries in
$C[x,x^{-1}]$ in  cases Q and M and $\tilde A,\tilde B\in C[x]$ in case S.
We do not need this statement in our proof.

\begin{prf} Consider the differential equation
\begin{eqnarray}\label{eq7a}
\delta(Y) &= &AY.
\end{eqnarray}
%
%
By { Lemma \ref{keyobs}}, 
it is equivalent to 
\begin{eqnarray}\label{eq8a}
 \delta Z &=& \mu^N\sigma^N(A) Z.
 \end{eqnarray}\renewcommand{\S}{{\mathcal S}}%
for any positive integer $N$. Let $s(x)=\sigma(x)\in C[x]$.  Let $\S$ be the set of finite
singularities of (\ref{eq7a}), except 0 in  cases Q and M and let $\S_N$ be the  {analogous} set for
(\ref{eq8a}).
Note that
{ $x_1\in\S_N$ if and only if} $\sigma^N(x_1)=s(s(\ldots s(x_1)\ldots ))$ is in $\S$.  
{ It can be verified in each of the cases S,Q and M that} for 
any finite singularity $\xi\in\S$ and
large enough $N$, there exists $x_1\in\S_N$ not in $\S$ satisfying 
$\sigma^N(x_1)=\xi$. Since (\ref{eq7a}) and (\ref{eq8a})
are equivalent, 
there exists a basis of solutions of (\ref{eq8a}) in $C\{x-x_1\}[(x-x_1)^{-1}]^n$.
Applying $\sigma^{-N}$ yields a basis of solutions of (\ref{eq7a}) in 
$C\{x-\xi\}[(x-\xi)^{-1}]^n$ in  cases S and Q. Substituting the branch 
$x=x_1\left(1+\frac{t-\xi}\xi\right)^{1/q^N}$ of $t^{1/q^N}$, we see that $\delta y (t) = A(t)y $ has a basis of solutions in 
$C\{t-\xi\}[(t-\xi)^{-1}]^n$ in case M as well.

Hence every finite singular point $\xi$ of (\ref{eq7a},) except 0 in cases Q and M, {is an apparent singular point.} 
\end{prf}
We now consider the behavior of solutions of equations~(\ref{eq1}) at infinity. 
In general, if one has a linear differential equation $\delta Y= A(x)Y $ with $A(x) \in \gl_n(C((x^{-1})))$, there exists a formal fundamental solution matrix of the form
\begin{eqnarray}\label{eq8c}
Y(x)& = &\Phi(x) x^L e^{Q(x)}
\end{eqnarray}
where $\Phi(x)$ is a formal power series in $x^{-1/r}$ for some integer $r$, $L$ is a constant matrix and $Q(x) = \sum_{j=1}^h Q_jx^{r_j}$ where the $Q_j$ are  diagonal matrices with entries in $C$ and the $r_j$ are positive rational numbers with $r_h > r_{h-1} > \ldots >r_1 >0$ (or $Q(x)\equiv 0$); furthermore $L$ and $Q(x)$ commute (c.f., \cite{BJL1979}). 

In the rest of the proof of Theorem~\ref{thm1}, we will treat cases Q and M together and then treat case S. 
\subsection{Proof of Theorem~\ref{thm1}: Cases Q and M.\label{th1sq}} We begin by showing that  under our hypotheses $Q(x)=0$ in these cases.

 \begin{lemma}\label{lem1a} Let $Y(x)$ be a formal fundamental solution matrix of $\delta(Y) = AY$ as in (\ref{eq8c}). 
In  cases Q and M we then have that $Q(x)=0$, \ie $\infty$ is a regular singular point.
  \end{lemma}
In the same way, it is shown that $0$ also must be a regular singular point in those two cases.\\[0.1in]
\noindent  {\bf Remark.} We note that Lemmas~\ref{lem1} and~\ref{lem1a} are sufficient to prove, in cases Q and M, that any  $f(x) \in \CX[[x]][x^{-1}]$ satisfying  (\ref{sysLS}) must be in $\CX(x)$.  Using the argument of Corollary~\ref{cor0}, we may assume that $f(x)$ is a component of a vector $y(x)$ of power series such that $y(x)$ is a solution  of  a consistent system (\ref{eq5}). Lemma~\ref{lem1a}  implies that $f(x)$ converges in a neighborhood of $0$ since this point is a regular singular point and Lemma~\ref{lem1} implies that $y(x)$ can be 
continued analytically to a meromorphic function on $\CX$. Finally, $y(x)$ has at most polynomial growth since $\infty$ is also a regular singular point. Altogether, we obtain that $y(x)$ has rational components and so $f(x)$ is rational.

  \begin{prf}Let $Q(x) = \diag(p_1(x), \ldots , p_n(x))$.  { By Lemma \ref{keyobs},} we have  that $B(x)Y(x)$ is a solution of $\delta(Y) = \mu \sigma (A(x)) Y(x)$ and $\sigma (Y(x))$ as well. Let $p(x)$ be a nonzero diagonal entry of $Q(x)$. By the uniqueness of the $p_i$ (Theorem 2,\cite{BJL1979}), we must have that for some $j$, $\sigma(p(x)) = p_j(x) $.  This implies that  the map $p(x) \mapsto\sigma( p(x))$ permutes the $p_i$ . From this we conclude that for some $m$, that 
$\sigma^m(p(x)) = p(x) $, a contradiction in  cases Q and M.  Therefore we have that $Q(x)=0$. 
 \end{prf}

\begin{lemma}\label{lemQM} In  cases  { {\rm Q}} and {{\rm M}} {there exists a}
matrix ${ T}\in { \gl_n(C)}$ such that $\lambda_1-\lambda_2\not\in\ZX^*$ for eigenvalues ${\lambda_1}, \lambda_2$ of ${ T}$ 
and  a gauge transformation $Z = F\,Y$, $F\in\GL_n(C(x))$, that transforms $\delta Y= A(x) Y$ into  $\delta(Z) = { T}Z$.

Moreover in  case {M}, the matrix ${ T}$ has the following properties:
\begin{enumerate}\item Its eigenvalues $\lambda$ are rational with smallest denominators
prime to $q$, $0\leq \lambda<1$.
\item There exists a diagonalisable matrix $D$ with integer eigenvalues
commuting with ${ T}$ such that ${ T}+D$ is conjugate to $q\,{ T}$.\end{enumerate}
\end{lemma} 
\begin{prf} Consider any local fundamental matrix $Y(x)$ of (\ref{eq7a}). Because of Lemma \ref{lem1} it can be continued analytically to a meromorphic function
on $\hat\CX$. Let $Y(x\,e^{2\pi i})$ be the fundamental matrix 
obtained by going once around $0$. There exists a constant invertible matrix $H\in\GL_n(\CX)$ 
such that 
$$Y(x\,e^{2 \pi i})=Y(x)H,$$
the so-called monodromy matrix associated to $Y$ and (\ref{eq7a}). It is well known that there exists a matrix 
${ T}\in\gl_n(\CX)$ such that $H=\exp(2\pi i { T})$.
Consider now $G(x)=Y(x)x^{-{ T}}=Y(x)\exp(-\log(x)\,{ T})$. By construction, $G(x)$ is invertible for every $x$ that is not a pole 
and $G(x\,e^{2 \pi i})=G(x)$, \ie\ $G$ is single valued.
By the previous lemma, $0$ and $\infty$ are regular singular points of (\ref{eq7a}) and therefore the growth of $G(x)$ as $x\to0$ or
$x\to\infty$ is at most polynomial (\ie\ there exists $K,L>0$ such that $|G(x)|\leq L(|x|+|x|^{-1})^K$).
Hence the meromorphic function $G$ on $\CX^*$ must be a  rational function,\ie\ $G\in\GL_n(\CX(x))$.
The transformation $y=G(x)z$ changes $\delta y=A(x)y$ into a system of differential equations
having $x^T$ as a fundamental solution, \ie\ $\delta z=T\,z$.
{The matrices $T$ and $G(x)$ satisfy $A = \delta(G) G^{-1} + GTG^{-1}, \ \det(G) \neq 0$.  Therefore the entries of $T$ and the coefficients of the entries of $G$ satisfy a finite number of equations and inequations with coefficients in $C$.  The Hilbert Nullstellensatz can be applied and yields that there also exists a solution in $C$. Therefore we may assume that $T\in\GL_n(C)$ and $G(x)\in\GL_n(C(x))$. A further constant gauge transformation allows us to assume that $T$ is in triangular form.  To insure that the eigenvalues of $T$ do not differ by nonzero integers, one can make a gauge transformation by a diagonal matrix $S$ whose diagonal entries are powers of $x$ (c.f., Lemma 3.11, \cite{PuSi2003}).}

In  case M, it also remains to show the stated properties. 
Observe that $B(x)Y(x)$ and $Y(x^q)$ are both fundamental solutions of $\delta { Y}=q\,A(x^q)\,{ Y}$
-- this was used before. Hence there exists a constant invertible matrix $D$ such that 
$$B(x)Y(x)=Y(x^q)\,D.$$ 
For the corresponding monodromy matrices this implies the relation
$H=D^{-1}\,H^q\,D$ because $x^q$ goes $q$ times around $0$ when $x$ goes once. Therefore
$H$ and $H^q$ are conjugate. As a first consequence, the mapping $\lambda\to\lambda^q$ must 
induce a permutation of the eigenvalues of $H$. Hence, for every eigenvalue $\lambda$ of $H$ there exists
a positive integer  $\ell$ such that $\lambda^{q^\ell}=\lambda$. The eigenvalues of
$H$ are therefore roots of unity and the smallest positive integer $m$ satisfying $\lambda^m=1$ must
be prime to $q$. This shows that the matrix ${ T}$ with $H=\exp(2\pi i { T})$ can be chosen 
having the first property. 

We can assume that ${ T}$ is in Jordan canonical form ${ T}=\mbox{\,diag\,}({ T}_1,...,{ T}_r)$  with 
Jordan blocks ${ T}_i$ of size $n_i$, say. 
The mapping $\lambda\mbox{\,mod\,}\ZX\to q\lambda\mbox{\,mod\,}\ZX$ induces a permutation of the equivalence classes
$\lambda\mbox{\,mod\,}\ZX$ of the eigenvalues $\lambda$ of ${ T}$. Hence there is a permutation matrix $P$ such that
the diagonal blocks of $P^{-1}\,q{ T}\, P$ have the same size and modulo $\ZX$ the same eigenvalues as 
those of ${ T}$. Therefore there is a diagonal matrix $D=\mbox{ \,diag\,}(d_1I_{n_1},...d_rI_{n_r})$ 
with integer $d_i$ such that the blocks of $P^{-1}\,q{ T}\, P$ and ${ T}+D$ have same size and 
same eigenvalues and thus the two matrices are conjugate. This proves the second property of
the matrix ${ T}$.~\end{prf}

\noindent We can now complete the proof of Theorem~\ref{thm1} in cases Q and M.  From the above lemmas we know that there is a gauge transformation that transforms $\delta(Y) = AY$ into a new equation $\delta(Y) = {T} Y$ where ${ T}$ has the stated properties. Apply the same gauge transformation to $\sigma(Y)=BY$ to yield $\sigma(Y)=\tilde BY$.
The consistency condition (\ref{eq2}) implies that
$$\delta\tilde B(x) = \mu { T}\tilde B(x) -\tilde B(x) { T}.$$
Comparing the orders of the poles, we see that $\tilde B(x)$ can have no {finite} poles other 
than 0 and we have $\tilde B(x)\in C[x,x^{-1}]$. 
Now let $\tilde B(x)=\sum_{m=-m_0}^{m_0}x^mB_m$.

In  case Q we obtain that ${ T} B_m -B_m { T}= m B_m$ for all $m=-m_0,.\ldots ,m_0$. As the eigenvalues of the operator
$X\mapsto { T} X -X { T}$ are exactly the differences  {{$\lambda_1-\lambda_2$}} where \ { $\lambda_1, \lambda_2$} are eigenvalues of ${ T}$
the condition for ${ T}$ shows that $B_m=0$ unless $m=0$. This proves the theorem in case Q.

In  case M we then have $q{ T} B_m -B_m { T}= m B_m$ for all $m=-m_0,...,m_0$. 
Hence $B_m\neq0$ can only hold for integers $m$
such that there exist eigenvalues {{$\lambda_1, \lambda_2$}} of ${ T}$ such that  $q\lambda_1-\lambda_2=m$. 
This shows the remaining statement of the theorem in  case M. \hfill  $\square$ 
\subsection{Proof of Theorem 1: Case S} \label{thm1m}
In the remaining case S, we consider again the 
formal fundamental solution matrix at infinity of the form (\ref{eq8c}).
We will first show that  under our hypotheses $r_h \leq 1$ in case S.
 \begin{lemma}\label{lem1b} Let $Y(x)$ be a formal fundamental solution matrix of $\delta(Y) = AY$ as in (\ref{eq8c}). We then have that $r_h \leq 1$ in case S.
  \end{lemma}
 \begin{prf}Let $Q(x) = \diag(q_1(x), \ldots , q_n(x))$. 
Let $q(x) = a_hx^{r_h} + a_{h-1}x^{r_{h-1}} + \ldots + a_1x^{r_1}$ be one of the $q_i$ and assume $r_h > 1$. Note that for any integer $m$ we have $q(x+m) = q(x) + mr_h x^{r_h - 1} +$  other terms of order less than  $r_h - 1$.  We have  that $B(x)Y(x)$ is a solution of $\delta(Y) = A(x+1) Y(x)$.  Therefore, by the uniqueness of the $q_i$ (Theorem 2,\cite{BJL1979}), we must have that for some $j$, $q(x+1) = q_j(x) \mbox{ modulo terms of order} \leq 0$.  This implies that  the map $q(x) \mapsto q(x+1)$ permutes the $q_i$ modulo non-positive terms. From this we conclude that for some $m$, that $q(x+m) = q(x) $ modulo terms of order less than $r_h-1$, a contradiction.  Therefore we have that $r_h \leq 1$. 
 \end{prf}
The following proposition  concerns only linear  differential equations  whose formal solutions satisfy the {conclusions of the Lemmas~\ref{lem1b} and \ref{lem1}}, with no mention of difference equations.
\begin{prop}\label{prop1}  Let $Y(x)$ be a formal fundamental solution matrix as in (\ref{eq8c}) of the differential equation  $\delta(Y) = AY, A \in \gl_n(C(x))$  with $r_h\leq 1$. Assume that all finite singular points of the differential equation are apparent in the sense of Lemma~\ref{lem1}. Then there exists a diagonal matrix  $\tilde{A}$ with constant entries and a gauge transformation $Z = F\,Y$ such that $Z$ satisfies $\delta(Z) = \tilde{A}Z$.
\end{prop} 

\begin{prf} For almost all directions $\theta$
there exist

\begin{enumerate}
\item sectors $S$ and $T$   of openings greater than $\pi$ bisected by $\theta$ and $\theta + \pi $, respectively,
\item a number $R>0$ and functions $\Phi_S(x)$ and $\Phi_T(x)$ analytic for $|x|>R$ in $S$ and $T$, respectively, such that
\begin{enumerate}
\item $\Phi_S(x)$ and $\Phi_T(x)$ are asymptotic to $\Phi(x)$ as $x \rightarrow \infty$ in their respective sectors,  and
\item $Y^S(x) =\Phi_S(x)x^L e^{Q(x)}$ and $Y^T(x) = \Phi_T(x)x^L e^{Q(x)}$ are solutions of $\delta Y = AY$.
\end{enumerate}
\end{enumerate}
{This can be proved using multisummation (c.f., Section~7.8, \cite{PuSi2003}) 
but can also be obtained independently (c.f., the proposition of \cite{Ju78}, p.85).}

We write $Q(x) = \diag(q_1(x)I_1, \ldots , q_s(x)I_s)$ with distinct $q_i(x)$ and $I_j$  identity matrices of an appropriate size.  We can furthermore order the $q_i(x)$ so that in the direction $\arg(x) = \psi = \theta + \pi/2$ we have $\Re(q_j(x))< \Re(q_k(x))$ for large $|x|$ if $j<k$. 
 Let us first assume that there are no monomials of the form $x^\alpha, 0<\alpha <1$ in any of the $q_i$ (we shall show below that this is indeed the case). We may then  write $Q(x) = \Lambda x$ where $\Lambda = \diag(\lambda_1I_1, \ldots , \lambda_sI_s)$ with  
$\Re(\lambda_jx)<\Re(\lambda_kx)$ for $j<k$ and $\arg(x) = \psi$. 
Let $C^+$ be the Stokes matrix defined by $Y^S(x) = Y^T(x) C^+$ in the component of the intersection of $S$ and $T$ that contains the line $\arg(x) = \psi$.  Because of the ordering of the $\lambda_i$ we have that $C^+$ is upper triangular with $1$ on the diagonal (\S 3,\cite{BJLII}; c.f., Theorem 8.13, \cite{PuSi2003}). 

    We include a short proof of this fact for the convenience of the reader. Write $Y^S=(Y^S_1|...|Y^S_s)$ and $Y^T=(Y^T_1|...|Y^T_s)$ in block columns according to the subdivision of $Q$ and do the analogous subdivision for  {$\Phi$}. Split $L=\diag(L_1,...,L_s)$ and write $C^+=(C^+_{i,j})_{i,j=1,...,s}$ in corresponding blocks.
First suppose that there is a non-vanishing block $C^+_{ij}$,$ i>j$, below the diagonal.
Fix one such $j$. Then 
\begin{equation}\label{stok+}
Y^S_j(x)=\sum_{i=1}^m Y^T_i C^+_{ij}
\end{equation} 
where $m\leq s$ is chosen as the last index $i$ such that $C^+_{i,j}\neq0$.
By assumption we have $m>j$.
Hence $Y^S_j \exp(-q_m(x))$ has a non-vanishing asymptotic expansion ${\Phi_m(x)}x^{L_m}C^+_{mj}$ as $x\to\infty$ on the line $\arg x=\psi$
contradicting the fact that it also has an expansion ${ \Phi_j(x)}x^{L_j}\exp(q_j(x)-q_m(x))$ and hence vanishes faster than any power of $x$
as  $x\to\infty$ on this line because of $m>j$. Therefore there is no non-vanishing block $C^+_{ij}$ below the diagonal.
In the same way one shows that the diagonal blocks $C_{jj}$ must equal $I_j$. This completes the proof that $C^+$ is upper triangular with 1 on the diagonal.

Similarly to $C^+$, let $C^-$ be the Stokes matrix defined by $Y^S(x) = Y^T(x) C^-$ in the component of the intersection of $S$ and $T$ that contains the line $\arg(x) = \psi-\pi$. Note that $\Re(\lambda_jx)>\Re(\lambda_kx)$ on the line $\arg(x) = \psi-\pi$ {if $j<k$}.  We therefore have that $C^-$ is lower triangular with $1$ on the diagonal.  
{ As all { finite }singular points are apparent, both $Y^S(x)$ and $Y^T(x)$ can be extended as meromorphic 
functions on $\CX$ with finitely many poles.}
We therefore have $C^+ = C^- = I$.  The matrix 
\[F(x) = Y^S(x)e^{-\Lambda x} = Y^T(x) e^{-\Lambda x}\]
has entries that are meromorphic with finitely many poles and of polynomial growth
at infinity, 
that is, rational entries. Therefore the transformation $Z = F^{-1}(x) Y$ yields the system $\delta(Y) = \Lambda Y$.

We now show that monomials of the form $x^\alpha, 0<\alpha <1$ cannot appear in any of the $q_i$. Assuming that this is not the case we will argue to a contradiction.  We write each $q_i(x)$ as 
\[ q_i(x) = \lambda_i x + \mbox{ terms involving $x^\alpha$ with $0<\alpha<1$.}\]
To fix notation, we assume that the same branches of $\log x$ are used to define $Y^S(x)$ and $Y^T(x)$ on the component of $S\cap T$ containing $\psi $ but that on the other component  we use $\arg x$ near $\psi-\pi$ in $S$ and $\arg x$ near $\psi+\pi$ in $T$.

We define the Stokes matrix $C^+$  as above.  Again we have that $C^+$ is upper triangular with $1$ on the diagonal. We divide this matrix again $C^+=(C^+_{ij})$ according to the diagonal blocks of $Q$.
We do not claim that $C^+$ is the identity matrix 
{ because different determinations of the powers are used in the definitions of
$Y^S(x)$ and $Y^T(x)$}
in the components of the intersection of $S$ and $T$ that {contain} the lines $\arg(x) = \psi\pm\pi$.
Nevertheless $Y^S(x)$ and $Y^T(x)$ are meromorphic on $\CX$ with finitely many poles
by the assumption of the Proposition so $Y^S(x) = Y^T(s)C^+$ also holds  in this component. {This implies that  $C^+_{ij}=0$ if $i<j$ and $\lambda_i \neq \lambda_j$ because otherwise
$\Re(\lambda_i x)> \Re(\lambda_jx)$ on the line $\arg(x) = \psi+\pi$, thus $\Re(q_i(x)) > \Re(q_j(xe^{-2\pi i}))$  for large $x$ on this line and therefore by (\ref{stok+}) the 
right hand side would grow faster as $|x|\to\infty$ on that line than the left hand 
side if $C^+_{ij}\neq0$, a contradiction.}

To complete the argument that no $x^\alpha, 0<\alpha <1$  appear in the $q_i(x)$, note the following:
\begin{enumerate}
\item[] If some $q_i(x)$ appears in $Q(x)$ we must have, for all integers $m$, that the conjugate $q_i(xe^{2m\pi i})$ also appears (\cite{BJL1979}, \S 1). 
\end{enumerate}
Continuing, we assume that some $q_i(x)$ contains a monomial $x^\alpha, 0<\alpha <1,$ and let $j$ be the smallest index for which this is true.  We claim that $\lambda_i \neq \lambda_j$ for $i < j$.  If not  then, for some $i$, $q_i(x) = \lambda x $ and $q_j( x) = \lambda x + c x^\alpha +\mbox{lower order terms}$. Among the conjugates $q_j(e^{2m\pi i}x)$, $m$ integer, we can find one such that $\Re(ce^{2m\pi i\alpha}x^{\alpha})$ tends to $-\infty$ as $x$ approaches infinity along the line $\arg(x) = \psi$. By the minimality of $j$, this is also the case for $m=0$. We then have $\Re(q_i(x)) > \Re(q_j(x)))$ eventually along this line as well, contradicting $i<j$. 
Hence the blocks $C^+_{ij}$ do not only vanish if $i>j$, but as seen above also if $i<j$ (since $\lambda_i\neq\lambda_j$ in this case).

Therefore  $Y_j^S = Y_j^T$ for the block column corresponding to $q_j(x)$.  This is impossible since near $\arg x=\psi+\pi$,  $Y_j^T(x)$ is asymptotic to something times $e^{q_j(x)}$ while $Y_j^S(x)$ is asymptotic to something times $e^{q_j(xe^{-2\pi i})}$.  This contradiction allows us to conclude that no fractional powers appear in the $q_i(x)$ and so completes the proof. \end{prf}
We can now complete the proof of Theorem \ref{thm1}  in  case S.  From Proposition~\ref{prop1} we know that there is a gauge transformation that transforms $\delta(Y) = AY$ into a new equation $\delta(Y) = \tilde{A}Y$ where $\tilde{A}$ is the  diagonal matrix $\diag(a_1, \ldots , a_n)$ with constant entries.  Apply the same gauge transformation to $\sigma(Y)=BY$ to yield $\sigma(Y)=\tilde BY$, $\tilde{B} = (b_{i,j})$.  Since the $a_i$ are constant,  the consistency condition (\ref{eq2}) implies
\begin{eqnarray}\label{eq9a}
\delta(b_{i,j})& = &(a_i-a_j)b_{i,j}
\end{eqnarray}
If $a_i \neq a_j$ then $b_{i,j} = 0$ since (\ref{eq9a}) then has no nonzero solution in $C(x)$. If $a_i = a_j$, then $\delta(b_{i,j}) = 0$. Therefore $\tilde{B}$ has constant entries.  Equation (\ref{eq2}) now implies that {$\tilde A$} and { $\tilde B$} commute.  This implies that there is a matrix $D\in \GL_n(C)$  which commutes with $\tilde{A}$ such that $D\tilde{B} D^{-1}$ is upper diagonal.  \hfill $\square$

\section{Reduction of systems of difference equations}\label{Sec2a} 
Here we present results analogous to Section \ref{Sec2} 
for systems of two difference equations with shifts having irrational quotient,
for systems of two  $q$-difference equations with ``independent'' $q$ and 
for systems of two Mahler equations with independent $q$.

We consider two commuting $C$-algebra 
{ endomorphisms} $\sigma_1,\,\sigma_2$ on $k$ extending 
the trivial automorphism on $C$, more specifically,
the three cases of couples $(\sigma_1,\sigma_2)$ below.
\begin{itemize}
\item[{\rm case 2S:}] Two shift operators $\sigma_j$ defined by
$\sigma_1(x)=x+1$ and $\sigma_2(x)=x+\alpha$ where $\alpha\in C\setminus\QX$.
\item[{\rm case 2Q:}] Two $q$-dilation operators $\sigma_j$, $j=1,2$, 
defined by $\sigma_j(x)=q_j\,x$ with multiplicatively independent\footnote{%
\ie\ there are no nonzero integers $n_j$ such that $q_2^{n_2}=q_1^{n_1}$.} 
$q_j\in C$, $|q_j|\neq0$, $q_j$ not a root of unity.
We also assume that
at least one of the $q_j$ does not have modulus 1, without loss of generality $|q_1|\neq1$.\footnote{
Alternatively to $|q_1|\neq1$, one can assume that $|q_1|=|q_2|=1$, $q_2$ not a root of unity, $q_1$ transcendental
over $\QX$ or  $q_1$ algebraic over $\QX$ such that its minimal polynomial has a root in $\CX$ of 
absolute value not equal to 1. See Remark \ref{Bez-ext}. } 
When considering $\sigma_j$ on the Riemann surface $\hat\CX$ of the logarithm, we fix
logarithms of $q_j$ used to determine $q_jx$ in $\hat\CX$ for given $x\in\hat\CX$. 
\item[{\rm case 2M:}] Two Mahler operators $\sigma_j$, $j=1,2$, defined by
$\sigma_j(x)=x^{q_j}$ with some multiplicatively independent positive integers $q_j$.
\end{itemize}

\noindent More precisely, we will consider systems 
\begin{equation}\label{eq2-1}
\begin{aligned}
\sigma_j(Y) & = & B_j\,Y,\ j=1,2
\end{aligned} 
\end{equation}
with $B_j \in \GL_n(k)$ that are {\it consistent}, that is $B_1$ and $B_2$ satisfy the  
consistency condition  given by
\begin{equation} \label{eq2-2}
\begin{aligned}
\sigma_1(B_2)B_1=\sigma_2(B_1)B_2.
\end{aligned}
\end{equation}
{As in Section  \ref{Sec2}, the consistency condition is closely related to the commutativity of 
$\sigma_1,\,\sigma_2$. Both are fundamental for our approach.  As (\ref{eq2}) did in Section \ref{Sec2}, 
the consistency condition guarantees that
$\sigma_1(\sigma_2(Z)) = \sigma_2(\sigma_1(Z))$ 
holds for any solution $Z$ of the system (\ref{eq2-1}) 
in any extension of $C(x)$. The other remarks following
(\ref{eq2}) apply analogously. }

We say that (\ref{eq2-1}) is {\it equivalent} (over $k$) to a system 
\begin{equation}\label{eq2-3}
\begin{aligned}
\sigma_j(Z) & = & \tilde{B_j}Z,\ j=1,2
\end{aligned} 
\end{equation}
with $ \tilde B_j \in \GL_n(k)$ if for some $G\in \GL_n(k)$,
\begin{equation}\label{eq2-4}
\begin{aligned}
\tilde{B_j} & =  \sigma_j(G)B_jG^{-1}, j=1,2
\end{aligned} 
\end{equation}
that is, if  (\ref{eq2-3}) comes from (\ref{eq2-1}) via the {\it gauge transformation} $Z = GY$. Note that the property of consistency is preserved under equivalence. In the present context we can prove
\begin{theorem}\label{thm2-1} In  cases 2S and 2Q, the system  (\ref{eq2-1}) satisfying the consistency condition (\ref{eq2-2})  
is equivalent over $k$  to a system (\ref{eq2-3})
with constant invertible commuting $\tilde B_j$, $j=1,2$, that is $\tilde B_j\in\GL_n(C)$, $j=1,2$,
and  $\tilde B_2\,\tilde B_1=\tilde B_1\,\tilde B_2.$ 

In  case 2M,  the system  (\ref{eq2-1}) satisfying the consistency condition (\ref{eq2-2})  
is equivalent over $K=C(\{x^{1/s}\mid s\in\NX^*\})$  to a system (\ref{eq2-3})
with constant invertible commuting $\tilde B_1$, $\tilde B_2$.
\end{theorem}

\noindent{\bf Remark:}   {1.\ When $n=1$ in case 2S, the result of this theorem 
is a reformulation of Lemma 3.1 in \cite{BH99} where the authors give  a purely algebraic proof of this special case.}\\[0.1in]
\noindent 2.\  In  case 2S, one of the equations can be made diagonal by
a (polynomial) transformation $Z=HU$, $H=\exp(Nx)$ with a certain nilpotent 
matrix $N$ commuting with its coefficient matrix.\\[0.1in]
\noindent 3.\ In  case 2M, the statement of the theorem also holds for $B_j\in\GL_n(K)$, $j=1,2$,
because fractional powers of $x$ can be removed by some change of variables $x=t^N$ with a suitable
positive integer $N$.\\[0.1in]
The proof of this theorem will be given for each of the cases separately in Sections~\ref{thm2S}, \ref{thm2Q} and \ref{thm2M}. As in Section \ref{Sec2}, we present a few consequences of the theorem before  presenting these proofs.
These concern a system 
\begin{equation}\label{sys2s}S_j(f(x))  =  \sigma_j^{m_j}(f(x)) + b_{j,m_j-1}(x) 
\sigma_j^{m_j-1}(f(x)) + \ldots + b_{j,0}(x)f(x) =  0,\ j=1,2
\end{equation}
with $b_{j,i}(x) \in \CX(x)$.

In view of the simplest nontrivial system of this form, $y(x+1)=y(x),\,y(x+\alpha)=y(x)$, in  case 2S with
non-real $\alpha$, it is
necessary to consider elliptic functions if we are interested in solving (\ref{sys2s}) using meromorphic
functions.
We recall the functions needed here (see \cite{DLMF}, section 23.2). The Weierstrass $\wp$-function
is the unique 1- and $\alpha$-periodic meromorphic function that has exactly one double pole in the basic
parallelogram with vertices $\pm1/2\pm\alpha/2$ and satisfies 
$\wp(x)=\frac1{x^2}+{\cal O}(x)$ as $x\to0$. 
All elliptic (\ie\ meromorphic 1- and $\alpha$-periodic)
functions can be expressed as rational functions of $\wp$ and $\wp'$.
The Weierstrass $\zeta$-function is the odd antiderivative of $\wp$. It satisfies 
$\zeta(x+1)=\zeta(x)+2\eta_1,\zeta(x+\alpha)=\zeta(x)+2\eta_2$ for all $x$, where
$\eta_j$ are certain constants such that the vectors $(\eta_1,\eta_2)$ and $(1,\alpha)$ are
linearly independent. The Weierstrass $\sigma$-function is the solution of
$\sigma'/\sigma=-\zeta$ with $\sigma'(0)=1$. It is an entire function
vanishing at the origin
and satisfies
$\sigma(x+1)=e^{2\eta_1x+\eta_1}\sigma(x)$, $\sigma(x+\alpha)=e^{2\eta_2x+\eta_2\alpha}\sigma(x)$
for $x\in\CX$. We introduce an additional function $\rho$ by
$\rho(\delta,x)=\sigma(x+\delta)/\sigma(x)$.\footnote{ This function appears already in
the classical works of \cite{ApLac, Herm}, but seems to be not so well known.} 
It is a meromorphic function
with a simple pole at the origin that satisfies
$$\rho(\delta,x+1)=e^{2\eta_1\delta}\rho(\delta,x)\mbox{ and }
\rho(\delta,x+\alpha)=e^{2\eta_2\delta}\rho(\delta,x).
$$
\begin{cor}\label{cor2s} Consider $\sigma_1$ and $\sigma_2$ as in  case 2S.

If $f(x)$ is a meromorphic function in $\CX$ that solves a system (\ref{sys2s}) and   $\alpha$ is nonreal,
then 
\begin{equation}\label{2s1}f(x)  = \sum_{i=1}^I\sum_{k=0}^{K} 
r_{i,k}(x)g_{i,k}(x)\zeta(x)^k e^{\alpha_i x}\rho(\delta_i,x)\end{equation}
where $\alpha_i,\delta_i\in\CX$, $r_{i,k}(x) \in \CX(x)$ and $g_{i,k}(x)$
are elliptic functions.
{ The latter and the 
functions $\zeta$ and $\rho$ are taken with respect
to the periods $1$ and $\alpha$.}
Conversely, any such function satisfies 
a pair of linear difference equations with rational coefficients.

If $f(x)$ is a meromorphic function in $\CX$ that solves a system (\ref{sys2s}) and
$\alpha$ is real or $f(x)$ has only finitely many poles, then
\begin{equation}\label{2s2} f(x)  = \sum_{i=1}^I r_i(x) e^{\alpha_i x} \end{equation}
where $\alpha_i \in \CX$ and $r_i(x) \in \CX(x)$. 

If $f(x)\in\CX[[x^{-1}]][x]$ satisfies a system 
(\ref{sys2s}) then $f(x)$ is rational.
\end{cor}
\noindent{\bf Remark:} {1.\ }In the case of real irrational $\alpha$, a slight extensions of the proof shows that the statement also holds
if $f(x)$ is a function on the real line continuous in all but finitely many points solving a system (\ref{sys2s}).
Indeed, it suffices to use the vector space $E$ of all functions on the real line continuous in all but finitely many points
and to apply Fej\'er's Theorem instead of the simple Fourier series.
This also implies that the given $f(x)$ is analytic at the points of continuity and
can be continued analytically to a meromorphic function with finitely many poles on the whole complex plane. 
Similar extensions can be obtained
in the second part of the subsequent corollary and in Corollary \ref{cor2m}. 
A similar reasoning is also crucial in the proof of the Theorem in  case 2M.\\[0.1in]
\noindent 2.\ In \cite{BG96}, B\'ezivin and Gramain consider {\it entire} solutions of (\ref{sys2s}) for case 2S under the assumption that $\alpha \in {\CX\backslash\RX}$. They show that such solutions must be of the form given in  (\ref{2s2}). They generalize this result to entire functions of $s$ variables satisfying $2s$ difference equations with respect to suitably independent multi-shifts.  The techniques are similar to those mentioned in the Remark following Corollary~\ref{cor0a}. In \cite{BH99}, Brisebarre and Habsieger replace the condition that $\alpha \in {\CX\backslash\RX}$ with $\alpha \in \CX\backslash\QX$ but need several other nontrivial technical conditions on the coefficients of (\ref{sys2s}) to show that entire solutions are of the form given in (\ref{2s2}). The techniques are essentially algebraic, reducing this problem to a similar problem for equations with constant coefficients in a manner different from our approach.  In \cite{Bezivin00}, B\'ezivin shows essentially the statement in remark 1, \ie\ that for $\alpha \in \RX\backslash \QX$, a solution $f(x)$, continuous  on $\RX$, of (\ref{sys2s}) is of the form given in (\ref{2s2}). Using properties of skew polynomial rings, B\'ezivin reduces the problem to the case of constant coefficient equations. {In \cite{Mart}, Marteau considers systems of scalar equations of the form $\sum_{i=0}^Na_i(x)f(x+\alpha_i) = 0$ where the $a_i(x)$ are polynomials and the $\alpha_i \in \CX$. He shows that, under certain restrictions on the $\alpha_i$, real valued continuous solutions and entire solutions of such systems must also be of the form given in  (\ref{2s2}).   Using essentially algebraic techniques, constant coefficients systems are considered in \cite{Jolly, Mart}.}

\begin{cor}\label{cor2q} Consider $\sigma_1$ and $\sigma_2$ as in  case 2Q.
If $\alpha:=\log(q_2)/\log(q_1)$ is nonreal and  
$f(x)$ is a meromorphic function on the Riemann surface
$\hat\CX$ of the logarithm 
and if $f(x)$ is a solution of a system (\ref{sys2s}) then
\begin{equation}\label{2q1}f(x)  = \sum_{i=1}^I\sum_{j=0}^J\sum_{k=0}^{K} 
r_{i,j,k}(x) \log(x)^j x^{\alpha_i} g_{i,k}(t)\zeta(t)^k\rho(\delta_i,t),\ t=\log(x)/\log(q_1),
\end{equation}
where $\alpha_i,\delta_i\in\CX$, $r_{i,j,k}(x) \in \CX(x)$ and $g_{i,k}(t)$
are elliptic functions. The latter and the 
functions $\zeta$ and $\rho$ are taken with respect
to the periods $1$ and $\alpha$.
Conversely, any such function satisfies 
a pair of linear $q$-difference equations with rational coefficients.

If $f(x)$ is a meromorphic function on $\hat\CX$ solving
a system (\ref{sys2s}) and $\alpha$ is real or $f(x)$ has only finitely many poles then
\begin{equation}\label{2q2} f(x)  = \sum_{i,j=0}^I r_{ij}(x) x^{\alpha_i}\log(x)^j 
\end{equation}
where $\alpha_i \in \CX$ and $r_{ij}\in\CX(x)$.

If $f(x)\in\CX[[x]][x^{-1}]$ satisfies a system 
(\ref{sys2s}) then $f(x)$ is rational.
\end{cor}

{\noindent{\bf Remark:} In \cite{BB92} B\'ezivin and Boutabaa use $p$-adic techniques to show that if $f(x)\in {F}[[x]][x^{-1}]$, where {\it {$F$} is the field of algebraic numbers}, satisfies a system 
(\ref{sys2s}) for case 2Q with $q_1, q_2 \in K$ multiplicatively independent, then $f(x)$ is rational. They prove a similar result for {$F$} any characteristic $0$ field assuming $q_1$ and $q_2$ are algebraically independent over $\QX$. In \cite{Bezivin00}, B\'ezivin shows that for $q_1,q_2$ multiplicatively independent positive real numbers, a solution of (\ref{sys2s}) that is continuous on the $]0,\infty[$ is of the form (\ref{2q2}). The proof again uses properties of skew polynomial rings and constant coefficient equations. }

In  case 2M, equations like $\log(\log(x^q))=\log(\log(x))+\log(q)$ yield interesting 
solutions. This suggests to consider the Riemann surface $\check\CX$ of $\log(\log(x))$. 
It is obtained by deleting the  point $1e^{i0}$, \ie\ the point with logarithm 0, from
the Riemann surface $\hat\CX$ of the logarithm and taking the universal covering of the remaining
manifold. It is biholomorphically mapped to $\hat\CX$ by $t=\log(x)$, biholomorphically
to $\CX$ by $s=\log(\log(x))$~.
\begin{cor}\label{cor2m} Consider $\sigma_1$ and $\sigma_2$ as in  case 2M.

If $f(x)$ is a meromorphic function on 
the universal cover of the open punctured unit disk $D(0,1)\setminus\{0\}$ (or on 
the universal cover of the annulus $\{x\in\CX\mid |x|>1\}$) solving
a system (\ref{sys2s}), then $f(x)$ can be continued 
analytically to a meromorphic function on ${\check\CX}$ and 
\begin{equation}\label{2m2} f(x)  = \sum_{i,j=0}^I r_{ij}(x) (\log(x))^{\alpha_i}\log(\log(x))^j 
\end{equation}
where $\alpha_i \in \CX$ and $r_{ij}\in\CX(\{x^{1/r}\mid r\in\NX^*\})$.
Conversely, any such function satisfies a pair of linear Mahler equations with rational coefficients.

If $f(x)$ is a meromorphic function on $\hat\CX$ solving a system (\ref{sys2s}) then
\begin{equation}\label{2m3} f(x)  = \sum_{j=-I}^I r_{j}(x) (\log(x))^j 
\end{equation}
where $r_{j}\in\CX(\{x^{1/r}\mid r\in\NX^*\})$.

If $f(x)\in\CX[[x]][x^{-1}]$ satisfies a system 
(\ref{sys2s}) then $f(x)$ is rational.
\end{cor}
{\noindent {\bf Remark:} As noted in the Introduction, the last statement of the above corollary was recently proved  by Adamczewski and Bell in \cite{AuB16}. Their tools include a local-global principle to reduce the problem to a similar problem over finite fields, Chebotarev's Density Theorem, Cobham's Theorem and  some asymptotics - all very different from the techniques used in the present work.

If one is only interested in proving the last statement, the application of Theorem \ref{thm2-1} in the 
proof of this
corollary can be replaced with the weaker Proposition \ref{2mprop}.  
This is discussed in the remarks following
Proposition \ref{2mprop}.\\

We now turn to the proof of the three corollaries.}

\begin{prf} {We will prove these three corollaries in parallel, diverging from this plan only when the cases force us to.}
We can assume without loss in generality that $b_{j,0}(x)\neq0$, $j=1,2$.
Otherwise in  cases 2S and 2Q, we can simply apply the inverses of $\sigma_1$
or $\sigma_2$. In  case 2M, we rewrite the system as a system
for a new function $\tilde f(x)=\sigma_1^{a}\sigma_2^{b}(f(x))$ with suitable positive integers
$a,b$, applying some powers of $\sigma_1$ or $\sigma_2$ to the equations. 
We then first obtain that $\tilde f(x)$ is as stated
in Corollary \ref{cor2m}. In the first two cases {of this corollary}, it follows immediately that $f(x)$ also has
the wanted form. In the last case, we obtain a series $f(x)\in\CX[[x]][x^{-1}]$ such that
$\sigma_1^{a}\sigma_2^{b}(f(x))=f(x^{q_1^aq_2^b})\in\CX(x)$. As shown at the end of the proof of
Corollary \ref{cor0}, this implies that $f(x)\in\CX(x)$. 

We proceed as for the corollaries concerning differential and difference equations.
The first part is the same for all the cases. Let $E$ be the vector space of all meromorphic
functions $f$ on $\CX$ for the first two cases of Corollary \ref{cor2s}, 
the vector space of  all meromorphic
functions $g$ on $\hat\CX$ for the first two cases of Corollary \ref{cor2q} and the second case
of Corollary \ref{cor2m}, and the vector space of
all meromorphic functions on $D(0,1)\setminus\{1\}$  in
the first case of Corollary \ref{cor2m}. Let
$E=\CX[[x^{-1}]][x]$  {or
$E=\CX[[x]][x^{-1}]$ or $E=\cup_{s\in\NX^*}\CX[[x^{1/s}]][x^{-{1/s}}]$ in the remaining cases
concerning formal power series, respectively}. 
Let $L=\CX(x)$ for Corollaries \ref{cor2s}, \ref{cor2q} and $L=K{=\CX(\{x^{1/r}\mid 
r\in\NX^*\})}$ for Corollary \ref{cor2m}.
In all cases, $\sigma_j$, $j=1,2$, 
are extended in the canonical way to automorphisms of $E$ and the extensions commute.

Consider the $L$-subspace $W$ of $E$ generated by $\sigma_1^m\sigma_2^n(f)$, 
$m=0,...,m_1-1,\ n=0,...m_2-1$. By (\ref{sys2s}), $W$ is invariant under $\sigma_j$ and
$\sigma_j^{-1}$; here the facts that $\sigma_j$ commute and that  $b_{j,0}(x)\neq0$ are used.

Let $w_1,...,w_s$ be an $L$-basis of $W$ and let $w=(w_1,...,w_s)^T$. Then we have
that 
\begin{equation}\label{sysW} \sigma_j(w)=B_j(x)w,\ j=1,2\end{equation}
with $B_j\in \GL_s(L)$ because the {components} of $\sigma_j(w)$ are again a basis of $W$.
The coefficient matrices of (\ref{sysW}) satisfy the consistency condition (\ref{eq2-2}). 
Indeed,
$$0=\sigma_1(\sigma_2(w))-\sigma_2(\sigma_1(w))=(\sigma_1(B_2)B_1-\sigma_2(B_1)B_2)w$$
and as the {components} of $w$ form a basis we obtain (\ref{sysW}).

Now we apply Theorem \ref{thm2-1} to the system $\sigma_j(Y)=B_j(x)Y,\ j=1,2$. 
It yields a gauge transformation $Z=GY$, $G\in\GL_n(L)$, that transforms the system to 
$\sigma_j(Z)=\tilde B_j Z,\ j=1,2$ where $\tilde B_j$ are constant commuting matrices.
The vector $z=Gw\in W^s$ satisfies $\sigma_j(z)=\tilde B_j z,\ j=1,2$.
Once we have proved that its components $z_1,...,z_s$ have the form desired in each of the
cases of the corollaries, the same will be true for $f(x)$ because the $z_i$ again form
a basis of $W$ -- only for the last part of Corollary \ref{cor2m} this has to be modified somewhat.

It remains to solve the system  
\begin{equation}\label{const}\sigma_j(z)=\tilde B_j z,\ j=1,2\end{equation} 
with constant invertible commuting $\tilde B_j$ in each of the spaces $E$.

{We now proceed as follows. The three corollaries contain nine cases in total. 
In each of these cases  we will consider functions meromorphic on a given domain or formal Laurent series and show that they are of the desired form.  Once we have finished this we will prove the converse statements.}

 We first consider  case 2S on the space of meromorphic functions on $\CX$.
If $\alpha$ is nonreal, then (\ref{const}) can be solved using the function $\rho$
and exponential functions.
Consider two commuting logarithms of $\tilde B_j$, \ie\ commuting matrices $L_j$
such that $\tilde B_j=\exp(L_j)$, $j=1,2$. 
Then there are uniquely determined (commuting) matrices $\Delta$ and $C$ such that
$$\begin{array}{rcl}
2\eta_1\,\Delta+C&=&L_1\\2\eta_2\,\Delta+\alpha C&=&L_2
\end{array}
$$  {where $(\eta_1, \eta_2)$ are as in the above definition of the $\zeta$-function.}
The matrix-valued function%
\footnote{See \cite{Ga}, chapter V, for the extension of entire functions to matrices.}
$Z(x)=e^{Cx}\rho(\Delta,x)=\frac1{\sigma(x)}e^{Cx}\sigma(x+\Delta)$
has meromorphic entries and
satisfies $Z(x+1)=e^Ce^{2\eta_1\Delta}Z(x)=\tilde B_1Z(x)$ and similarly 
$Z(x+\alpha)=\tilde B_2Z(x)$. If $z$ is a vectorial solution of (\ref{const}), then
$c(x)=Z(x)^{-1}z(x)$ is 1- and $\alpha$-periodic and meromorphic. Hence its components
are elliptic functions.

In order to express $\sigma(x+\Delta)$ using scalar functions we proceed as it is well known for
the exponential. Let $T$ be invertible such that $J=T^{-1}\Delta T$ is in Jordan canonical form
and let $J=D+N$ with diagonal $D$ and nilpotent $N$, $D=\diag(d_1,...,d_s)$.
Then 
$$\sigma(x+\Delta)=T\sigma(x+D)\left(I+\sum_{k=1}^{s-1}\frac1{k!}\sigma(x+D)^{-1}
   \sigma^{(k)}(x+D)N^k\right)T^{-1}\ ;$$
here $\sigma(x+D)=\diag(\sigma(x+d_1),...,\sigma(x+d_s))$ and similarly $\sigma(x+D)^{-1}\sigma^{(k)}(x+D)$ is the diagonal matrix
of the quotients $\frac{\sigma^{(k)}(x+d_j)}{\sigma(x+d_j)}$ which in turn 
can be expressed using powers of $\zeta(x+d_j)$ and elliptic functions.
Finally as $\zeta(x+d_j)-\zeta(x)$ are elliptic, $\frac1{\sigma(x)}\sigma(x+\Delta)$ can be expressed using
the functions $\rho(d_j,x)$, powers of $\zeta(x)$ and elliptic functions.
  
If $\alpha$ is real and irrational or the number of poles of $f$ is finite, 
then we consider again commuting logarithms $L_j$ of the coefficient matrices
$\tilde B_j$.
Put $c=\exp(-L_1 x)z$. Then $c$ satisfies $\sigma_1(c)=c$ and $\sigma_2(c)=\tilde Bc,$  $\tilde B = \tilde B_2 e^{-\alpha L_1}$.
Then $c$ cannot have a pole at all, because otherwise by the two equations all points
of the lattice $\{ k+m\alpha|k,m\in\ZX\}$ would be poles and hence 
the set of poles would be dense for real irrational $\alpha$ or infinite otherwise.

As $c$ is 1-periodic and entire, it can be expanded in a Fourier series on $\CX$.
On some horizontal strip $S$ of finite width, we have the uniformly convergent series
$$c(x)=\sum_{k=-\infty}^\infty c_k e^{2\pi i k x},\ x\in S.$$
The second equation satisfied by $c$, \ie 
$$c(x+\alpha)=\tilde Bc(x)\mbox{ for }x\in  S$$
implies that $\tilde Bc_k=\exp(2\pi i k \alpha)c_k$
for $k\in\ZX$. If $c_k\neq0$ for some $k$ then $\exp(2\pi i k \alpha)$ is an eigenvalue of
$\tilde B$. As there are only finitely many such eigenvalues, the above Fourier series
can contain only a finite number of terms. 
This proves that 
$c(x)=\sum_{k=-k_0}^{k_0} c_k e^{2\pi i k x},\ x\in \CX$
with some positive integer $k_0$. As a consequence $z=\exp(L_1 x)c$ is of the desired form.

On $E=\CX[[x^{-1}]][x]$, we can assume that $\tilde B_1$ is diagonal according
to a remark following the theorem.
Hence it suffices to show that the solution of $\sigma_1(z(x))=d\,z(x)$ in $E$ 
is a constant for every complex number $d$.
Write such a solution  as 
$z(x)=\sum_{k=k_0}^{\infty}c_k\,x^{-k}$
with some $k_0\in\ZX$ and $c_k\in \CX$, $k\geq k_0$. Unless $z=0$, we can assume that
$c_{k_0}\neq0$. Comparing the coefficients of $x^{-k_0}$, we find that
$d\, c_{k_0}=c_{k_0}$. Hence $d=1$. Comparing the coefficients
of $x^{-k_0-1}$ we find that $ c_{k_0+1}=c_{k_0+1}-c_{k_0}k_0$. Hence $k_0=0$.
The new series $\tilde z(x)=z(x)-c_0$ is again a solution of   $\sigma_1(\tilde z(x))=\tilde z(x)$
but unless $\tilde z(x)=0$, the corresponding series 
$\tilde z(x)=\sum_{k=k_1}^{\infty}c_k\,x^{-k}$ has a positive $k_1$ which is impossible as seen above. 
Hence $\tilde z(x)=0$ and $z$ is a constant.

In  case 2Q and in the space of meromorphic functions on $\hat\CX$, we simply put
$t=\log(x)/\log(q_1)$ and consider $z$ as  a {vectorial} function  of $t$ now. 
The system $\sigma_j(z)=\tilde B_j z,\ j=1,2$ is then transformed
into a system $z(t+1)=\tilde B_1z(t)$, $z(t+\alpha)=\tilde B_2z(t)$,  
for a function $z$ meromorphic on
$\CX$. The above considerations then prove that {the entries of } $z$ {have} the desired form 
in the cases of nonreal and real $\alpha$. The only difference is that in the expansion of
$x^C=e^{C\log(q_1)t}$ using scalar functions, logarithms of $x$ may occur.  

In  case 2Q, $E=\CX[[x]][x^{-1}]$, we write a solution of 
$\sigma_j(z(x))=\tilde B_j\,z(x)$, $j=1,2$,
as a series  $z(x)=\sum_{k=k_0}^{\infty}c_k\,x^{k}$
with some $k_0\in\ZX$ and $c_k\in \CX^s$, $k\geq k_0$.
We obtain that each $c_k$ satisfies $\tilde B_jc_k=q_j^kc_k$,
$j=1,2$. Therefore unless $c_k=0$ for some $k$, the numbers $q_j^k$
must be eigenvalues of $\tilde B_j$ for $j=1,2$. 
As there are only finitely many eigenvalues we obtain that  
{$z(x)\in (\CX[x,x^{-1}])^s\subset(\CX(x))^s$}.

In  case 2M and in the space of meromorphic functions on $D(0,1)\setminus\{1\}$, we put
$s=\frac{\log(\log(x))}{\log(q_1)}$ and consider $z$ as a function of $s$ now.
Proceeding as above \footnote{The restriction of the domain
causes no additional problem. } with the resulting system of difference equations
with irrational $\alpha=\frac{\log(q_2)}{\log(q_1)}$, the result follows readily. {The result concerning meromorphic functions on the annulus
$\{x \in \CX \ | \ |x|>1\}$  is reduced to the one on the
punctured unit disk by the change of variables $x\to1/x$}.
Observe that in this case, the reduction to a constant system uses a gauge transformation
with coefficients in $\CX(\{x^{1/r}\mid r\in\NX^*\})$.

In  case 2M and in the space of meromorphic functions on $\hat\CX$, we have to
find solutions of the  system $\sigma_j(z(x))=\tilde B_j z(x),\ j=1,2$, 
that are meromorphic in $\hat\CX$. {
  Letting $t = \log x$ we can reduce this problem to   the search for solutions, meromorphic in $\CX$, of a system with constant coefficients  in case 2Q. One then considers the series expansion  at the origin of such solutions and concludes, as in case 2Q, that these vectors have entries in $\CX[t,t^{-1}]$. Therefore   the entries of $z(x)$ involve powers of $\log x$.}

In  case 2M, $f(x)\in\CX[[x]][x^{-1}]$, we have to solve 
$\sigma_j(z(x))=\tilde B_j\,z(x)$, $j=1,2$,
in $E=\cup_{s\in\NX^*}\CX[[x^{1/s}]][x^{-{1/s}}]$. Expanding
a solution in a series immediately proves that $z(x)$ is constant.
Therefore $f(x)$ is an element of $K$. We write
$f(x)=g(x^{1/s})$ with some rational function $g\in\CX(x)$ and some positive integer
$s$. As in the proof of Corollary \ref{cor0}, this yields $f(t^s)=g(t)$
and therefore also $f(t^s)=g(\xi t)$ for any $s$-th root of unity $\xi$.
Hence $f(t^s)=\frac1s \sum_{\xi^s=1}g(\xi t)\in\CX(t^s)$. This yields
that $f(x)$ is rational.

For the converse concerning the shift operators,  
consider first $f(x)=g(x)\zeta(x)^je^{\beta x}\rho(\delta,x)$
where $\beta,\delta\in\CX$, $j$ is some positive integer and
$g$ is an elliptic function. By the properties of $\rho$ and the exponential,
we obtain with $a_1=e^{\beta+2\eta_1\delta}$ that 
$(\sigma_1-a_1)^{j+1}(f(x))=0$. Again products of solutions of $\sigma_1$-difference equations
with rational functions also satisfy $\sigma_1$-difference equations and sums of 
solutions of $\sigma_1$-difference equations also satisfy some $\sigma_1$-difference equation.
Therefore $f(x)$ given by (\ref{2s1}) also does.
We obtain the $\sigma_2$-difference equation for $f(x)$ in the same way.

For the proof of the converse concerning two $q$-difference equations, we proceed analogously.
A slight change is that we have to consider 
$$f(x)=\log(x)^ix^\beta g(t)\zeta(t)^j\rho(\delta,t),\ t=\log(x)/\log(q_1).$$ 
Here $(\sigma_1-a_1)^{i+j+1}(f(x))=0$ for $a_1=q_1^\beta e^{2\eta_1\delta}$ and
$(\sigma_2-a_2)^{i+j+1}(f(x))=0$ for $a_2=q_2^\beta e^{2\eta_2\delta}$. 

For the proof of the converse concerning two Mahler systems, consider
$$f(x)=x^{r}(\log(x))^{\alpha}(\log(\log(x)))^j$$ with $r\in\QX$, $\alpha\in\CX$
and $j\in\NX$. There are $k\in\NX^*$, $m\in\NX$ such that $(q_1^k-1)q_1^mr=:\ell$ is an integer.
Then 
$$(q_1^{-k\alpha}x^{-\ell}\sigma_1^k- id)^{j+1} \sigma_1^m(f(x))=0.$$ 
It follows as before that $f(x)$ given by (\ref{2m2}) satisfies a $\sigma_1$-Mahler
equation. For the second equation concerning $\sigma_2$, we proceed analogously.
\end{prf}

We now turn to the proof of Theorem~\ref{thm2-1}. 
The consistency condition (\ref{eq2-2}) can be interpreted as follows. 
{ \begin{remark}\label{obs2sigma}\rm The gauge transformation
$Z=B_2 Y$ transforms 
\begin{equation}\label{b1}\sigma_1(Y)=B_1\,Y\end{equation} 
into the equivalent system $\sigma_1(Z)=\sigma_2(B_1)\,Z$.
Iterating this procedure, we find that the system (\ref{b1}) is equivalent to the systems 
\begin{equation}\label{b1N}\sigma_1(U)=\sigma_2^N(B_1)\,U\end{equation} 
for any positive integer $N$ by the gauge transformation $Z=G_NY$,
$G_N=\sigma_2^{N-1}(B_2)...\sigma_2(B_2)B_2$. Thus, if $W$ is a fundamental solution matrix of 
(\ref{b1}) then   both $\sigma_1^N(W)$ and $G_NW$ are fundamental solution matrices of (\ref{b1N}).
\end{remark}}

The rest of the proof of Theorem \ref{thm2-1} will be given separately for the three cases.

{\subsection{Proof of Theorem~\ref{thm2-1}: Case 2S}\label{thm2S}}

In a first step, we consider analytic continuation of solutions of (\ref{b1}) 
{under the hypotheses of the Theorem}.\\

\begin{lemma}\label{2s-app}
In  case 2S, consider a strip $S=\{x\in\CX|a_1<\Im x<a_2\}$, $-\infty\leq a_1<a_2\leq+\infty$.
If $g(x)$ is a holomorphic solution of (\ref{b1}) for $x\in S$ with sufficiently large 
positive real part {or with sufficiently large negative real part} then $g(x)$ can be 
continued analytically to a meromorphic function in $S$ with finitely many poles.
\end{lemma}
\begin{prf} Let $\calM$ denote the set of $x_1\in\CX$ such that $B_1(x)$ or $B_1(x)^{-1}$
has a pole at $x=x_1$.
Consider a solution $g$ holomorphic for $x\in S$ with large {positive} real part.
By the difference equation (\ref{b1}), it can be continued analytically to $S$ except
for possible poles in $(\calM-\NX)\cap S$. { By Remark \ref{obs2sigma},} for $N\in\NX$, 
the function $g_N=G_Ng$ is a solution
of $\sigma_1(U)=\sigma_2^{N}(B_1)U$ holomorphic for $x\in S$ with large real part.
By its difference equation, it can be continued analytically to $S$ except for possible
poles in $(\sigma_2^{-N}(\calM)-\NX)\cap S=(\calM-\NX-N\alpha)\cap S$, because $x_1$ is a pole of 
$\sigma_2^N(B_1^{\pm1})$ if and only if
$\sigma_2^N(x_1)$ is a pole of $B_1^{\pm1}$. This implies that $g=G_N^{-1}g_N$ 
can also be analytically continued
to $S$ except possible poles in $(\sigma_2^{-N}(\calM)-\NX)\cap S$ and in $\calN_N\cap S$, where
$\calN_N$ is the (finite) set of poles of  $G_N^{-1}$. 
Therefore $g$ can be {continued} analytically to $S$ with the exception
of $(\calM-\NX-N\alpha)\cap(\calM-\NX)\cap S$ and $\calN_N\cap S$.

We claim that the former intersection is empty for appropriate $N$.
The set $\{d-c+\ZX\mid c,d\in\calM\}$ is finite since $\calM$ is finite. As $\alpha$ is irrational,
the set $\{N\alpha+\ZX\mid N\in\NX\}$ is infinite. Hence {we can select} $N\in\NX$ such that
for all $c,d\in\calM$, the difference $d-c\not\equiv N\alpha\mod\ZX$.
If the first intersection is nonempty, there exists $h\in S$ such that
$h=d-N\alpha-m$ and $h=c-n$, where $c,d\in\calM$ and $m,n\in\NX$. This implies
that $d-c-N\alpha=m-n\in\ZX$, a contradiction.    Thus the intersection is indeed empty.
This means that $g$ can be continued analytically to $S\setminus\calN_N$, where
$\calN_N$ is finite.

The proof in the case of $g$ analytic for $x\in S$ with large negative real part is
analogous.~\end{prf}

Concerning the behavior at infinity, it is known that $\sigma_1(Y)=B_1\,Y$ has a formal fundamental
solution 
\begin{equation}\label{ffs-dz}Y(x)=\Phi(x)x^L\,e^{Q(x)}x^{Dx},
\end{equation}
where $\Phi(x)$ is a formal power series in $x^{-1/r}$ for some integer $r$, $L$ is a constant matrix with eigenvalues $\gamma_j$
satisfying $0\leq\Re(\gamma_j)<\frac1r$, $Q(x) = \sum_{j=1}^h Q_jx^{r_j}$ where the $Q_j$ are  diagonal matrices with entries in $C$ and the $r_j$ are positive rational numbers 
with $1=r_h > r_{h-1} > \ldots >r_1 >0$ (or $Q(x)\equiv 0$) and
$D$ is diagonal with entries in $\frac1r\ZX$; furthermore $L$, $Q(x)$ and $D$ commute ({\it c.f.,} \cite{Im84}, chapter I, and \cite{PuSi}, section 6.1).
The leading term $Q_h$ of $Q(x)$ is chosen such that the imaginary parts of its
entries are between 0 and $2\pi$\footnote{Recall that solutions of
$\sigma_1$-difference equations remain solutions when multiplied by
1-periodic functions.}, $2\pi$ excluded.
We can write $D=\diag(d_\ell I_\ell,\ell=1,...,m)$ and  
$Q(x)=\diag(q_\ell(x)I_\ell,\ell=1,...,m) $ with identity matrices of appropriate size {$n_\ell \times n_\ell$} and distinct couples 
$(d_\ell,q_\ell(x))$, $\ell=1,...,m$. Then also $L=\diag(L_1,...,L_m)$ with diagonal
blocks of corresponding size. The formal fundamental solution is essentially unique, \ie\ 
except for a permutation of the diagonal blocks and passage from some $L_\ell$ to a conjugate matrix.

{ By Remark \ref{obs2sigma}, both} $B_2(x)Y(x)$ and 
$\sigma_2(Y(x))=Y(x+\alpha)$ are formal fundamental solutions of 
$\sigma_1(Z)=\sigma_2(B_1)\,Z$. Re-expanding we find that
$Y(x+\alpha)=\tilde\Phi(x)x^{\tilde L}e^{Q(x)}x^{Dx},$
where $\tilde L\equiv L+\alpha D\mod\frac1r\ZX$. {Also writing} 
$\tilde L=\diag(\tilde L_\ell,\ell=1,...,m)$ we obtain
from the essential uniqueness of the formal fundamental solution
that for each $\ell$, the matrices $\tilde L_\ell\equiv L_\ell+\alpha d_\ell\mod\frac1r\ZX$ 
and $L_\ell$ are conjugate.\footnote{
Observe that there can be no permutaion of diagonal blocks because the couples $(d_\ell,q_\ell(x))$, $\ell=1,...,m$
are distinct.} Now if $a_k+\frac1r\ZX,\,k=1,...,r_\ell$ are the equivalence
classes of the eigenvalues of $L_\ell$ modulo $\frac1r\ZX$,
then $a_k+\alpha d_\ell+\frac1r\ZX$ are those of $\tilde L_\ell$. Hence
the mapping $x+\frac1r\ZX\mapsto x+\alpha d_\ell+\frac1r\ZX$ induces a permutation of the
equivalence classes of the eigenvalues of $L_\ell$. Applying it several  {times}, if necessary, to some
eigenvalue $a_k$, we obtain the existence of some positive integer $N$ such that
$a_k+\frac1r\ZX=a_k+N\alpha d_\ell+\frac1r\ZX$. Since $\alpha$ is not rational, but $d_\ell$ is,
this is impossible unless $d_\ell=0$ and we obtain that $D=0$. The difference equation
(\ref{b1}) and by symmetry also $\sigma_2{Y}=B_2 Y$ are hence {\it mild}
in the sense of \cite{PuSi}, section 7.1.

Next we show the statement analogous to Proposition \ref{prop1}.
\begin{lemma}\label{prop2-1} In the present context, there exists a diagonal matrix  $\tilde{B}_1$ with constant 
entries and a gauge transformation $Z = F\,Y$, $F\in\GL_n(k)$, such that $Z$ satisfies 
$\sigma_1(Z) = \tilde{B}_1Z$.\end{lemma}
\begin{prf}This is an adaptation of the proof of {Proposition}
\ref{prop1}. For $\theta\neq0$ sufficiently close to 0, there exist
\begin{enumerate}
\item sectors $S$ and $T$   of openings greater than $\pi$ bisected by $\theta$ and $\theta + \pi $, respectively,
\item a number $R>0$ and functions $\Phi_S(x)$ and $\Phi_T(x)$ analytic for $|x|>R$ in $S$ and $T$, respectively, such that
\begin{enumerate}
\item $\Phi_S(x)$ and $\Phi_T(x)$ are asymptotic to $\Phi(x)$ as $x \rightarrow \infty$ in their respective sectors,  and
\item $Y^S(x) =\Phi_S(x)x^L e^{Q(x)}$ and $Y^T(x) = \Phi_T(x)x^L e^{Q(x)}$ are solutions of $\sigma_1 Y = AY$.
\end{enumerate}
\end{enumerate}
This can be proved using multisummation 
(c.f., \cite{BrFa96}, \cite{PuSi}, section 9.1).

We write $Q(x) = \diag(q_1(x)I_1, \ldots , q_s(x)I_s)$ with distinct $q_j(x)$ and $I_j$  identity matrices of an appropriate size.  
We split
\[ q_i(x) = \lambda_i x + \mbox{ terms involving $x^\alpha$ with $0<\alpha<1$.}\]
As $0\leq \Im\lambda_i<2\pi$, we may assume that $\theta$ is so close to 0
that 
\begin{equation}\label{diffs}
\max_{j,k}(\Re(\lambda_jx)-\Re(\lambda_kx))< {-}\Re( 2\pi i x) \mbox{ if }\arg(x) = \psi=\theta+\pi/2.
\end{equation}
\noindent{Indeed, this is equivalent to ${\max_{j,k}}\,\Im\left((\lambda_j-\lambda_k)e^{i\theta}\right)<  2\pi \cos(\theta)$
which is true for sufficiently small~$|\theta|$.}

According to Lemma \ref{2s-app} both $Y^S(x)$ and $Y^T(x)$ can be continued analytically
to meromorphic functions on $\CX$ with finitely many poles. {The same is true for 
their inverses, because their transposed inverses also satisfy equations to which  Lemma \ref{2s-app} applies.} 
We call these extensions again $Y^S(x)$, $Y^T(x)$.

Let $D(x)$ be the connection matrix defined by $Y^S(x) = Y^T(x) D(x)$. It is
meromorphic on $\CX$ with finitely many poles and 1-periodic. Hence it is entire. Moreover,
as $Y^S(x)$, $Y^T(x)$ and  $(Y^T)^{-1}(x)$, it has at most exponential growth
as $|\Im x|\to\infty$ in the intersections of $S$ and $T$.
Hence the Fourier series of $D(x)$ has at most finitely many terms.
Let us write $D(x)=\sum_{k=-k_0}^{k_0}D^{(k)} e^{2\pi i k x}$. 
The blocks  of $D(x)$ corresponding to the subdivision of $Q(x)$ 
are denoted by $D_{i,j}(x)$, those of $D^{(k)}$ by $D^{(k)}_{i,j}$.
Write $Y^S=(Y^S_1|...|Y^S_s)$ and $Y^T=(Y^T_1|...|Y^T_s)$ in corresponding
block columns.

We claim that $D(x)$ is constant.
To show this, fix some $j$. Then 
$$Y^S_j(x)=\sum_{m=1}^s\sum_{k=-k_0}^{k_0} Y^T_m(x)e^{2\pi i k x} D^{(k)}_{m,j}.$$
Because of (\ref{diffs}) we obtain that
$D^{(k)}_{m,j}=0 \mbox{ if }  k<0. $ 
Otherwise the right hand side would grow more rapidly as $|x|\to\infty $,
$\arg(x)=\psi$ than the left hand side.
For the line $\arg(x)=\psi+\pi$,  the exponentials $e^{2\pi i k x}$ are ordered
inversely with respect to growth as $|x|\to\infty$.
Here we find that
$D^{(k)}_{m,j}=0\mbox{ if } k > 0.$
Altogether we obtain that $D(x)$ is constant.

The rest of the proof is identical to the one of Proposition \ref{prop1} {with $D(x)$ replacing
the matrix $C^+$ in the latter argument}.\end{prf}
We can now complete the proof of Theorem~\ref{thm2-1}{ in case 2S}. Apply the gauge transformation $Z=FY$ of Lemma
\ref{prop2-1} to $\sigma_2 Y= B_2 Y$ to yield $\sigma_2 Z= \tilde B_2 Z$,
$\tilde B_2=(b_{i,j})$.
Since $\tilde B_1=\diag(a_1,...,a_n)$ is constant diagonal, the consistency
condition (\ref{eq2-2}) implies that
$$\sigma_1(b_{i,j})=\frac{a_i}{a_j}b_{i,j}.$$
If $a_i\neq a_j$ we obtain that $b_{i,j}=0$ since the above equation has no
nonzero solution in $C(x)$ then. If $a_i=a_j$ we have $\sigma_1(b_{i,j})=b_{i,j}$
and hence $b_{i,j}$ is a constant. Therefore $\tilde B_2$ is also constant.
This completes the proof in  case 2S.\vspace{3mm}

{\subsection{Proof of Theorem~\ref{thm2-1}: Case 2Q}\label{thm2Q}}
In  case 2Q, we omit the index 1 of $q_1$ and assume $|q|>1$ for simplicity. Fix a logarithm $2\pi i\tau$
of $q$. Concerning analytic continuation we prove here\\

\begin{lemma}\label{2q-app}
In  case 2Q, consider a spiraling strip $S=\{e^{2\pi i \tau t}\in\hat\CX|a_1<\Im t<a_2\}$, $-\infty\leq a_1<a_2\leq+\infty$.
If $g(x)$ is a holomorphic solution of (\ref{b1}) for $x\in S$ with sufficiently large 
modulus {or  with sufficiently small modulus} then $g(x)$ can be 
continued analytically to a meromorphic function in $S$ such that the projections of its poles to $\CX^*$
form a finite set.
\end{lemma}
\begin{prf} The proof is analogous to that of Lemma \ref{2s-app}. {The only  sets of possible poles are now
$(\calM\cdot q^{-\NX})\cap S$ or   $(\calM\cdot q_2^{-N}\cdot q^{-\NX})\cap S$, respectively, } and their intersection
is empty for convenient $N$ because of the condition imposed on $q$ and $q_2$ in  case 2Q. We leave it to the reader
to fill in the details.
\end{prf}
Concerning the behavior at $0$ (and similarly at $\infty$), it is known that there exists a formal gauge transformation
$Z=G\,Y$, $G\in\GL_n(\CX[[x^{1/s}]][x^{-1/s}])$, {$s\in\NX^*$,} that reduces (\ref{b1})
to a system $\sigma(Z)=x^DA_0Z$, where $D$ is a diagonal matrix with entries in $\frac1s\ZX$ and $A_0\in\GL_n(\CX)$ 
such that any eigenvalue $\lambda$ of $A_0$ satisfies $1\leq|\lambda|<|q|^{1/s}$, moreover $D$ and $A_0$ commute.
If we write $D=\diag(d_1I_1,...,d_rI_r)$ with distinct $d_j$ and $I_j$ identity matrices of an appropriate size,
then $A_0=\diag(A_0^{1},...,A_0^r)$ with diagonal blocks $A_0^{j}$ of corresponding size. 
$D$ and $A_0$ are essentially unique, \ie\ except for a permutation of the diagonal blocks and passage from
some $A_0^j$ to a conjugate matrix. If $D$ happens to be 0, then $s$ can
be chosen to be 1 and $G$ is convergent
(see \cite{PuSi}, ch.\ 12, \cite{Ad31}, \cite{Car12}).

{ According to Remark \ref{obs2sigma}}, our system (\ref{b1}) is equivalent to
$\sigma(U)=\sigma_2(B_1)\,U$. The gauge transformation $V=\sigma_2(G)U$ now transforms this system to
$\sigma(V)=\sigma_2(x^DA_0)V$. Now $\sigma_2(x^DA_0)=x^{D}\,q_2^{D}A_0$ and there is a diagonal matrix $F$
with entries in $\frac1s\ZX$ commuting with $D$ and $A_0$ such that the gauge transformation $W=x^{F}V$ reduces
the latter system to $\sigma(V)=x^D\tilde A_0V$, where $\tilde A_0=q^{-F}q_2^DA_0$ has again eigenvalues
with modulus in $[1,|q|^{1/s}[$. Now we write $\tilde A_0=\diag(\tilde A_0^{1},...,\tilde A_0^r)$ and 
fix some $j\in\{1,...,r\}$. If $a^{j}_1,...,a^{j}_{r_j}$ are the eigenvalues of $A_0^j$ then
$q^{-f_j}q_2^{d_j}a^j_\ell$, $\ell=1,...,r_j$, are those of $\tilde A_0^j$. By the uniqueness of the reduced form, 
the mapping $x\mapsto q^{-f_j}q_2^{d_j}x$ induces a permutation of the eigenvalues of $A_0^j$. 
If we apply it several times, if necessary, we obtain the existence of some $\ell\in\{1,...,r_j\}$ and of some positive integer
$k$ such that $q^{-kf_j}q_2^{kd_j}a^j_\ell=a_\ell^j$. Due to our condition on $q$ and $q_2$ this is only possible
if $d_j=0$. Thus we have proved that $D=0$.

Hence $0$ and $\infty$ are regular singular points of (\ref{b1}). There is a gauge transformation
$Z=G_0(x) Y$, $G_0(x)\in\GL_n(\CX\{x\}[x^{-1}])$ reducing the system to $\sigma(Z)=A_0Z$ and a gauge transformation
$V=G_\infty(x) Y$, $G_\infty(x)\in\GL_n(\CX\{x^{-1}\}[x])$ reducing the system to $\sigma(V)=A_\infty V$,  where
$A_0,\,A_\infty$ are constant invertible matrices with eigenvalues in the annulus 
$1\leq|\lambda|<|q|[$.

Now we fix a matrix $L_0$ such that $A_0=q^{L_0}(=\exp(2\pi i \tau L_0)))$ and thus $Y_0(x)=G_0(x)^{-1}x^{L_0}$ is a solution
of (\ref{b1}) in some neighborhood of 0 in $\hat\CX$. 
By Lemma \ref{2q-app}, this solution can be continued analytically to a meromorphic function on $\hat\CX$
such that the projections of its poles to $\CX^*$ form a finite set. This implies that $G_0(x)^{-1}$ 
can be continued to a meromorphic function on $\CX$ with finitely many poles. We use the same name for this
extension. In some annulus $K<|x|<\infty$, $K$ sufficiently large, it can be expanded in a convergent Laurent
series 
$$G_0(x)^{-1}=\sum_{m=-\infty}^{\infty}G_m x^m$$
with some constant matrices $G_m$ and therefore also
$$G_\infty(x)G_0(x)^{-1}=\sum_{m=-\infty}^{\infty}P_m x^m$$
with constant matrices $P_m$. 

Now by construction, $G_\infty(x)G_0(x)^{-1}x^{L_0}$ is a solution of the 
equation $\sigma(V)=A_\infty\,V$. Using $q^{L_0}=A_0$, this implies the equations
$$P_m\,q^mA_0=A_\infty P_m,\ m\in\ZX.$$
Now if $m\neq0$, then  $q^mA_0$ and $A_\infty$ cannot have a common eigenvalue as those of  $A_0$,  $A_\infty$ have
a modulus in $[1,|q|[$. Hence $P_m=0$ for $m\neq0$ and there is a constant {invertible}
matrix $P_0$ such that $G_\infty(x)=P_0\,G_0(x)$ for $|x|>K$.  In particular, the above Laurent series
of $G_0(x)^{-1}$ has only finitely many terms corresponding to positive powers of $x$ since this is the case for 
$G_\infty(x)^{-1}$. This means that $\infty$ is only a pole of the $G_0(x)^{-1}$, a meromorphic function
on $\CX$ with finitely many poles. We obtain that $G_0(x)^{-1}$ and hence also $G_0(x)$ have entries that are rational 
functions. 

Thus we have shown that (\ref{b1}) is equivalent over $\CX(x)$ to $\sigma(Z)=A_0 Z$ with some constant invertible 
matrix $A_0$. {To complete the proof of Theorem~\ref{thm2-1} for case 2Q,} apply the same gauge transformation to $\sigma_2(Y)=B_2(x)Y$ to obtain
$\sigma_2(Z)=\tilde B_2(x)Z$. The consistency condition implies here that
$$\sigma_1(\tilde B_2(x))A_0=A_0\tilde B_2(x).$$ 
Expanding $\tilde B_2(x)$ in a convergent Laurent series in a punctured disk centered at the origin
$$\tilde B_2(x)=\sum_{m=-M}^\infty C_m x^m$$
we find that the coefficients satisfy the equations $C_m\,q^mA_0=A_0C_m$, $m=-M,...$. Again $q^mA_0$ and $A_0$ have no common eigenvalue
and hence $C_m=0$ if $m\neq0$. Thus $\tilde B_2(x)$ is also a constant and the theorem is proved in  case 2Q as well.

{
\begin{remark}\label{Bez-ext} {\rm Following ideas of B\'ezivin ( \cite[Page 90]{Ramis92}) we can establish Theorem \ref{thm2-1}, case 2Q under the assumptions that $|q_1| = |q_2| = 1$ , $q_1,q_2$ multiplicatively independent, $q_2$ not a root of unity, $q_1$ transcendental over $\QX$
or  $q_1$ algebraic over $\QX$ such that its minimal polynomial has a root in $\CX$ of 
absolute value not equal to 1. For the proof, let $F\subset \CX$ be the field generated by $q_1$ and $q_2$ and the coefficients of the entries of $B_1$ and $B_2$.  This is a finitely generated extension of $\QX$. Hence by the assumption on $q_1$, there is an embedding $\psi: F \rightarrow \CX$ such that such that $|\psi(q_1)| \neq 1$.  We can extend $\psi$ to an automorphism of $\CX$ which we denote again by $\psi$.  After applying $\psi$ to the system (\ref{eq2-1}), we get a new system that satisfies the hypotheses of Theorem 12.  Therefore there exists a gauge transformation $Z=GY$ transforming this new system to  a system with constant invertible commuting matrices $\tilde{B_1}, \tilde{B}_2$.  One sees that 
the system with matrices $\psi^{-1}(\tilde{B_1}), \psi^{-1}(\tilde{B}_2)$ satisfies the conclusion of Theorem \ref{thm2-1} with respect to our original system.}\end{remark}
}

{\subsection{Proof of Theorem~\ref{thm2-1}: Case 2M}\label{thm2M}}

In  case 2M, it is more convenient to use different notation. We consider\\

\begin{equation}\label{2msys}y(x^p)=A(x)y(x),\ \ \ y(x^q)=B(x)y(x)\end{equation}
with multiplicatively independent positive integers $p$ and $q$ and
$A(x),B(x)\in\GL_n(C(x))$ satisfying the consistency condition
\begin{equation}\label{2mcons}A(x^q)B(x)=B(x^p)A(x).\end{equation}
We want to show that there exist $G(x)\in \GL_n(K)$ and commuting $A_0,B_0\in\GL_n(C)$ such that
the gauge transformation $y=G(x)z$ reduces (\ref{2msys}) to 
\begin{equation}\label{2msysz}z(x^p)=A_0z(x),\ \ \ z(x^q)=B_0z(x).\end{equation}
The proof is more involved than in  cases 2S and 2Q because there are  several fixed
points of the mapping $x\mapsto x^p$, namely $0,\infty$ { and the $(p-1)$-th roots of unity,}
because solutions
holomorphic in some neighborhood of the origin { or $\infty$ can only be extended
using (\ref{2msys}) to the unit disk or the annulus $|x| >1$, respectively,} and because the behavior
of the solutions of Mahler systems near 0 and $\infty$ is not well understood. The consistency
condition (\ref{2mcons}) is crucial and will be used many times.

The plan is as follows. In a first step we prove 
 that any  formal vectorial solution of a system  (\ref{2msys}) satisfying  (\ref{2mcons}) is 
rational and deduce the statement of the theorem  under the additional hypothesis that
$x=0$ is a {\it regular singular point} of $y(x^p)=A(x)y(x)$, \ie\ there is a formal series
$G_0(x)\in \GL_n(\hat K),\ \hat K=\bigcup_{r\in\NX^*}C[[x^{1/r}]][x^{-1/r}]$, such that the gauge transformation
$y=G_0(x)z$ reduces the equation to one with a  constant coefficient matrix. 
In a second step, we prove
that $x=0$ is always regular singular for a consistent system (\ref{2msys}), (\ref{2mcons})
thus completing the proof.\\

\begin{prop}\label{2mprop}
In  case 2M,  consider the system  (\ref{2msys}) satisfying the consistency condition (\ref{2mcons})
and suppose that $g(x)\in (C[[x]][x^{-1}])^n$ is a formal vectorial solution. Then
$g(x)\in C(x)^n$.\end{prop}
\begin{cor}\label{2mcor}
In  case 2M,  consider the system  (\ref{2msys}) satisfying the consistency condition (\ref{2mcons})
and suppose that
the point $x=0$ for the first equation of (\ref{2msys}) is regular singular.
Then (\ref{2msys}) is equivalent over $K=C(\{x^{1/s}\mid s\in\NX^*\})$  to a system (\ref{2msysz})
with constant invertible commuting $A_0$ and $B_0$.
\end{cor}
\noindent {\bf Remark: }The Proposition could be deduced from the last part of Corollary \ref{cor2m}, \ie\ from Theorem 1.1 
of \cite{AuB16}. 
Indeed, consider the $C(x)$-subspace space of $C[[x]][x^{-1}]$ generated
by the components of $g(x)$. By (\ref{2msys}), it is invariant under $\sigma_1,\sigma_2$ and
it follows as usual that each  component of $g(x)$
satisfies a system of two scalar linear  $p$- and $q$-Mahler equations and hence is rational
by the Theorem 1.1 of \cite{AuB16}.

Conversely, Theorem 1.1 of \cite{AuB16} can be deduced from Proposition \ref{2mprop}:
Given a formal solution $f(x)$ of a system (\ref{sys2s}), the first part of the proof of Corollary \ref{cor2m}
constructs a system (\ref{2msys}) satisfying the consistency condition (\ref{2mcons}) having a solution vector
in $(\CX[[x]][x^{-1}])^n$. As this solution vector is actually a basis of some vector space containing
$f(x)$, we can assume that one of its components is $f(x)$. Proposition \ref{2mprop} then yields that
$f(x)$ is rational. 

Observe that
Theorem \ref{thm2-1} in case 2M (and hence also the first part of Corollary \ref{cor2m})
is not an immediate consequence of the result of \cite{AuB16} because 
it is not clear a priori that the point 0 is regular singular under the hypotheses of the theorem.
This statement is the contents of Proposition \ref{2mregsing}.\\[0.1in]
\noindent{\bf Proof of Corollary \ref{2mcor}.}
We consider a formal series $G_0(x)\in \GL_n(\hat K)$ reducing  $y(x^p)=A(x)y(x)$ to
$z(x^p)=A_0z(x)$ with a constant invertible matrix $A_0$. This means
\begin{equation}\label{2mg0}G_0(x^p)=A(x)G_0(x)A_0^{-1}.\end{equation}
By a change of variables $x=t^r$, if necessary, we can assume that 
$G_0(x)\in \GL_n(\hat k)$, $\hat k=C[[x]][x^{-1}]$ (we do not use that $A(x), B(x)$ are in 
$C(x^r)$ then).

Applying the gauge transformation $y=G_0(x)z$ to the second equation $y(x^q)=B(x)y(x)$,
we obtain $z(x^q)=\tilde B(x)z(x)$ with some $\tilde B(x)\in\GL_n(\hat K)$ satisfying the consistency
condition $A_0\,\tilde B(x)=\tilde B(x^p)A_0$. Using the series expansion of $\tilde B(x)$, it is readily shown
that $\tilde B(x)$ must be constant and commutes with $A_0$. 
We write $\tilde B(x)=:B_0$. Thus $G_0(x)$ also satisfies
\begin{equation}\label{2mg0b}G_0(x^q)=B(x)G_0(x)B_0^{-1}.\end{equation}

The system (\ref{2mg0}), (\ref{2mg0b}) can be considered as a vectorial system 
$$Y(x^p)=\bar A(x)Y(x),\ \ Y(x^q)=\bar B(x)Y(x),$$ 
where $Y(x)$ has coefficients in $gl_n(\hat k)\sim \hat k^{n^2}$ and $\bar A(x)$, $\bar B(x)$ are the matrices
of the linear operators mapping $Z$ to $A(x)ZA_0^{-1}$ {and  $B(x)ZB_0^{-1}$, respectively}.
This system satisfies the consistency condition, because $A(x)$ and $B(x)$ do and $A_0,B_0$ commute.
Now Proposition \ref{2mprop} can be applied to its formal solution $G_0(x)$ and  we obtain that
$G_0(x)$ has rational entries.  \hfill $\square$\\[0.1in]
{\bf Proof of Proposition \ref{2mprop}.} We first show that $g(x)$ is actually convergent.
This could be deduced from \cite{Ran92}, section 4.
For the convenience of the reader, we provide a short proof.

To do that, we truncate $g(x)$ at a sufficiently high power of $x$ to obtain
$h(x)\in \gl_n(C[x][x^{-1}])$ and introduce $r(x)=h(x)-A(x)^{-1}h(x^p)$
and $\tilde g(x)=g(x)-h(x)$. Then we have
\begin{equation}\label{2m-gr}
\tilde g(x)=A(x)^{-1}\tilde g(x^p) - r(x).\end{equation}
We denote the valuation of $A(x)^{-1}$ at the origin by $s\in\ZX$
and introduce $\tilde A(x)=x^{-s}A(x)^{-1}$ which is holomorphic at the origin.

First choose $M\in\NX$ such that $pM+s>M$ and $h(x)$ such that $g(x)-h(x)$ has at least
valuation $M$. Then by (\ref{2m-gr}), $r(x)$ also has at least valuation $M$.
Now consider $R>0$ such that $\tilde A(x)$ is holomorphic and bounded on $D(0,R)$. 
Then consider for positive $\rho<\min(R,1)$ the vector space $E_\rho$ of all series  
$F(x)=\sum_{m=M}^\infty F_mx^m$ such that $\sum_{m=M}^\infty|F_m|{\rho}^m$ 
converges and define the norm
$|F(x)|_\rho$ as this sum. Then $E_\rho$ equipped with $|\ |_\rho$ is a Banach space
and the existence of a unique solution of (\ref{2m-gr}) in $E_\rho$
for sufficiently small $\rho>0$ follows from the Banach fixed-point theorem using
that $|x^s F(x^p)|_\rho\leq \rho^{Mp+s-M}|F(x)|_\rho$ for $F(x)\in E_\rho$.
This proves the convergence of $\tilde g(x)$ and hence of $g(x)$.

By {(\ref{2msys}), rewritten $g(x)=A(x)^{-1}g(x^p)$}, the function $g$ can only be extended
analytically to a meromorphic function on the unit disk. 
{ According to Theorem 4.2 of \cite{Ran92} (see also \cite{BCR13}), it is sufficient to show that $g(x)$ {does not have} the
unit circle as a natural boundary and the rationality of $g(x)$ follows. {We show how it follows naturally, in our context,} that $g(x)$ can be continued 
analytically as a meromorphic function to all of $\CX$ and, as well, that it has only finitely many poles.
The rationality of $g(x)$ then follows as in  \cite{Ran92} and \cite{BCR13} from a growth estimate.}

As we want to extend $g(x)$ {beyond}
the unit disk, we use the change of variables $x=e^t$, $u(t)=y(e^t)$ and obtain a system of
$q$-difference equations
\begin{equation}\label{2msys2q}u(pt)=\bar A(t)u(t),\ \ \ u(qt)=\bar B(t)u(t)\end{equation}
with $\bar A(t)=A(e^t)$, $\bar B(t)=B(e^t)$. It satisfies the consistency condition
\begin{equation}\label{2mconsi2q}\bar A(qt)\bar B(t)=\bar B(pt)\bar A(t).\end{equation}
We are not in  case 2Q, however, because $\bar A(t),\bar B(t)$ are not rational in $t$, but rational 
in $e^t$.

Nevertheless, the local theory at the origin {used in the proof of Lemma~\ref{2q-app}} applies to the present
consistent system. {In particular, we know that the origin is a regular singular point in case 2Q.} We {therefore} obtain a matrix $A_1$ with eigenvalues $\lambda$ in the annulus $1\leq|\lambda|<p$
and $G_1(t)\in\GL_n(C\{t\}[t^{-1}])$ such that $u=G_1(t)v$ reduces the first equation of (\ref{2msys2q})
to $v(pt)=A_1v(t)$. This means 
\begin{equation}\label{2mg1}G_1(pt)=\bar A(t)G_1(t)A_1^{-1}\mbox{ for small }t.\end{equation}
Applying the same gauge transformation to the second equation of (\ref{2msys2q})
yields an equation $v(qt)=\tilde{\bar B}(t)v(t)$ with some $\tilde{\bar B}(t)\in\GL_n(C\{t\}[t^{-1}])$.
It satisfies the consistency condition $A_1\tilde{\bar B}(t)=\tilde{\bar B}(pt)A_1$.
As at the end of the proof in case 2Q, we expand $\tilde{\bar B}(t)=\sum_{m=m_0}^\infty C_m t^m$.
The coefficients satisfy $A_1C_m=C_m (p^mA_1)$, $m\geq m_0$. As $A_1$ and $p^mA_1$ have no common eigenvalue
unless $m=0$, we obtain that $\tilde{\bar B}(t)=:B_1$ is constant and commutes with $A_1$.
We note the second equation satisfied by $G_1$  
\begin{equation}\label{2mg1b}G_1(qt)=\bar B(t)G_1(t)B_1^{-1}\mbox{ for small }t.\end{equation}

\begin{lemma}\label{2mcont}
The functions $G_1(t)^{\pm1}$ 
can be continued analytically to meromorphic functions on $\CX$
and there exists $\delta>0$ such that both  can be continued analytically
to the sectors $\{t\in\CX^*\mid \delta<\arg(\pm t)<2\delta\}$.
\end{lemma}
{\bf Proof of the lemma.} 
Let $\calM$ be the set of poles of $\bar A(t)^{\pm1}$, \ie\ the set of $t$
such that $e^t$ is a pole of $A(x)$ or $A(x)^{-1}$. Note that $\calM$ is $2 \pi i$-periodic,
has no finite accumulation point and is contained in some vertical strip
$\{t\in\CX\mid -D<\Re t<D\}$.
By (\ref{2mg1}), $G_1(t)^{\pm1}$ can be continued analytically to $\CX^*\setminus(\calM\cdot p^\NX)$
and thus to meromorphic functions on $\CX$ which we denote by the same name. By construction,
$G_1(t)^{\pm1}$ are also analytic in some punctured neighborhood of the origin.
By the properties of $\calM$, the infimum of the $|\Re t_1| $ on the set of all $t_1\in\calM$
having nonzero real part is a positive number. As $\calM$ is contained in some vertical strip
there exist sectors $\{t\in\CX^*\mid \delta<\arg(\pm t)<2\delta\}$ disjoint to $\calM$
and hence to $\calM\cdot p^\NX$. Therefore $G_1(t)^{\pm1}$ can be analytically continued to these sectors
and the lemma is proved.~\hfill$\square$\vspace{2mm}

Consider now the function $d(t)=G_1(t)^{-1}g(e^t)$. By Lemma \ref{2mcont} and because $g(x)$
is holomorphic in some punctured neighborhood of $x=0$, $d(t)$ is defined
and holomorphic for some sector
$S=\{t\in\CX\mid |t|>K,\ \pi+\delta<\arg t<\pi+2\delta\}$. By (\ref{2msys}), (\ref{2mg1}),
and (\ref{2mg1b}) it satisfies
\begin{equation}\label{2md}d(pt)=A_1d(t),\ \ d(qt)=B_1d(t)\mbox{ for }t\in S.\end{equation}
To solve (\ref{2md}), consider a matrix 
$L_1$ commuting with $B_1$ such that $p^{L_1}=A_1$. Put $F(t)=t^{-L_1}d(t)$.
Then
\begin{equation}\label{2mf}F(pt)=F(t),\ \ F(qt)=\tilde B_1F(t)\mbox{ for }t\in S\end{equation}
where {$\tilde B_1=B_1q^{-L_1}$}. Thus $H(s)=F(e^s)$ is $\log(p)$-periodic on the half-strip
$B=\{s\in\CX\mid \Re s>\log(K),\ \pi+\delta<\Im s<\pi+2\delta\}$ and can be expanded in a Fourier series.
This implies that
\begin{equation}\label{2mfourier}F(t)=\sum_{\ell=-\infty}^\infty F_\ell\,t^{\frac{2 \pi i}{\log(p)}\ell}\mbox{ for }t\in S.
\end{equation}
The second equation of (\ref{2mf}) yields conditions on the Fourier coefficients
$$F_\ell\exp\left(2\pi i\tfrac{\log(q)}{\log(p)}\ell\right)=\tilde B_1F_\ell\mbox{ for }\ell\in\ZX.$$
Therefore $F_\ell=0$ unless $\exp\left(2\pi i\frac{\log(q)}{\log(p)}\ell\right)$ is an eigenvalue of 
$\tilde B_1$.
Since $p$ and $q$ are multiplicatively independent, the quotient $\frac{\log(q)}{\log(p)}$ is irrational and hence
$\exp\left(2\pi i\frac{\log(q)}{\log(p)}\right)$ is not a root of unity. Therefore all the numbers
$\exp\left(2\pi i\frac{\log(q)}{\log(p)}\ell\right)$, $\ell\in\ZX$ are different and only finitely many of them
can be eigenvalues of $\tilde B_1$. This shows that the Fourier series  (\ref{2mfourier}) has finitely many terms
and thus $F(t)$ can be analytically continued to the whole Riemann surface $\hat \CX$ of $\log(t)$. The same holds
for $d(t)=t^{L_1}F(t)$. 

{ Since $g(x)$ is a convergent Laurent series in $C[[x]][x^{{-1}}]$, the function 
$h(t)=g(e^t)$ is holomorphic for $t$ with large negative real
part and $2\pi i$-periodic.
We conclude using Lemma \ref{2mcont} that $h(t)=G_1(t)d(t)$ can be
analytically continued to a meromorphic function on $\hat \CX$, in particular the point 
$t=2\pi i$ is at most a pole of $h$. By its periodicity,
this implies that $t=0$ also is at most a pole of $h$ and that it can be continued analytically 
to a $2\pi i$-periodic meromorphic function on $\CX$ which we denote by the same name.

This periodicity allows one to define a meromorphic function $\tilde g(x)$ on $\CX\setminus\{0\}$
by $\tilde g(e^t)=h(t)$. As $\tilde g(x)=g(x)$ for small $|x|\neq0$ by the construction of $h$,
we have shown that $g(x)$ can be continued analytically to a meromorphic function on $\CX$
which we again will denote by the same name.

The formula {$h(t)=G_1(t)d(t)$} and Lemma \ref{2mcont} also imply that $h(t)$ is analytic in some 
sector
$\tilde S=\{t\in\CX^*\mid \delta<\arg t<2\delta\}$ with small positive $\delta$. 
As this sector contains some half strip
$\{t\in\CX\mid \Re t>L,\mu\Re t<\Im t<\mu\Re t+3\pi\}$ 
for some positive $L,\mu$ which has vertical width larger than $2\pi$ and $h$ is $2\pi i$-periodic,
its poles are contained in some vertical strip $\{t\in\CX\mid -L<\Re t<L\}$.
For the function $g(x)$ this means that it can be continued analytically to a meromorphic function on $\CX$
with finitely many poles.
}

The proposition is proved once we have shown that $g(x)$ has polynomial growth as $|x|\to\infty$. This is done as in
the proof of Theorem 4.2 in \cite{Ran92} { (see also \cite{BCR13})}. Consider
$r_0{>1}$ such that $g(x)$ and $A(x)$ are holomorphic on the annulus $|x|>r_0/2$. There are
positive numbers $K,M$ such that $|A(x)|\leq K|x|^M$ for $|x|\geq r_0$.
Consider now the annuli
$$\calA_j=\{x\in\CX\mid r_0^{p^j}\leq|x|<r_0^{p^{j+1}}\},\ j=0,1,...$$
covering the annulus $|x|\geq r_0$. Any $x\in\calA_j$ can be written $x=\xi^{p^j}$ with some $\xi\in\calA_0$.
Then we estimate using (\ref{2msys}) and the inequality for $|A(x)|$
$$|g(x)|=|g(\xi^{p^j})|\leq K ^j \left(|\xi|^{p^{j-1}}\cdots|\xi|^p|\xi|\right)^M \max_{r_0\leq|\xi|\leq r_0^p}|g(\xi)|.$$
Hence there is a positive constant $L$ such that
$|g(x)|\leq L\, K^j \,|x|^{\frac M{p-1}}\mbox{ for }x\in\calA_j.$
Assuming $\log(r_0)\geq1$ without loss in generality, we find that $j\leq \log(\log(|x|))/\log(p)$ for $x\in\calA_j$.
Hence there exists $d>0$ such that
$$|g(x)|\leq L\, (\log(|x|))^d\,|x|^{\frac M{p-1}}\mbox{ for }|x|>r_0$$
and the proof of Proposition \ref{2mprop} is complete.   \hfill $\square$

{
For later use, we note another corollary of Proposition \ref{2mprop}.
\begin{cor}\label{2mpoly}Consider multiplicatively independent positive integers $p,q$. Let $A(x)$ and $B(x)$ 
be two matrices with polynomial entries such that the consistency condition (\ref{2mcons})
holds and $A(0)$ and $B(0)$ are invertible. 
Then there exists a matrix $G(x)$ with polynomial entries, $G(0)=I$ and
$$A(x)=G(x^p)\,A(0)\,G(x)^{-1},\ \ B(x)=G(x^q)\,B(0)\,G(x)^{-1}.$$\end{cor}}
\noindent As an application, consider two polynomials $a(x),b(x)$ without constant term satisfying
$$a(x^q)-a(x)=b(x^p)-b(x).$$
This is the consistency condition for the matrices 
$A(x)=\left(\begin{matrix}1&0\\a(x)&1\end{matrix}\right)$ and $B(x)=\left(\begin{matrix}1&0\\b(x)&1\end{matrix}\right)$.
The above Corollary yields the existence of a polynomial matrix $G(x)$ with the stated properties.
The conditions encoded in  $A(x)G(x)=G(x^p)$ and $G(0)=I$ imply that $G(x)=\left(\begin{matrix}1&0\\g(x)&1\end{matrix}\right)$
with some polynomial $g(x)$ without constant term. We obtain 
$$a(x)=g(x^p)-g(x),\ \ b(x)=g(x^q)-g(x).$$ 
 In the case of $p,q$ without common divisor, the existence of such a polynomial $g(x)$ can be proved
directly. Under the present hypothesis that $p$ and $q$ are multiplicatively independent, {this is possible but less
elementary.} A similar reasoning will appear in the proof of Corollary \ref{2mregsing}.\\[0.1in]
{\bf Proof of Corollary \ref{2mpoly}.} 
Putting $G(x)=I+H(x)$, the first equation of the statement is equivalent to
$$H(x)=A(x)^{-1}A(0)-I+A(x)^{-1}H(x^p)A(0).$$
Using the fixed point principle in $\calM=\gl_n(x\,C[[x]])$ equipped with the {$x$-adic norm}, it follows that
the latter equation has a unique solution in $\calM$. 
This also follows from Proposition 34 of \cite{Ro15}.
Hence there is a unique formal series $G(x)=I+...$ such that $A(x)=G(x^p)\,A(0)\,G(x)^{-1}$.
Now put $\tilde B(x)=G(x^q)^{-1}B(x)G(x)$. As $A(x), B(x)$ satisfy the consistency condition,
so do $A_0{:=A(0)}$ and $\tilde B(x)$. Again a series expansion yields that $\tilde B(x)=:B_0$ is
constant.

As in the proof of Corollary \ref{2mcor}, the formal solution $G(x)\in I+\calM$ of 
the system  
\begin{equation}\label{sysG}{G(x^p)=A(x)G(x)A_0^{-1},\ \ G(x^q)=B(x)G(x)B_0^{-1}}\end{equation}
can be considered as a formal (matricial) solution of a consistent system (\ref{2msys}) and
thus must have rational entries by Proposition \ref{2mprop}. 

{It remains to prove that $G(x)$ actually has polynomial entries. As $A(x)$ has polynomial
entries, the first equation of (\ref{sysG}) implies the following statement: If $x=x_0$
is a finite regular point of all entries of $G(x)$ then so is $x=x_0^p$. The contrapositive -- which is also true -- says: 
If $x=x_0$ is a finite pole of some entry of $G(x)$ then so is any $p$-th root of $x_0$. Now there is no finite subset of $\CX$
that is stable with respect to taking any $p$-th root of any element besides the empty set and $\{0\}$. 
As the entries of $G(x)$ have no poles at $x=0$, the set of their finite poles must be empty. Hence $G(x)$
has polynomial entries.\hfill $\square$\vspace{2mm}
}

Unfortunately not every Mahler system is regular singular.\\
{\bf Example.} The system $y(x^p)=A(x)y(x),\ 
A(x)=\left(\begin{matrix}1&0\\x^{-1}&2\end{matrix}\right)$
is not regular singular at the origin. Otherwise, there exist an invertible matrix
$T(x)\in\GL_2(E), E=\CX[[x^{1/r}]][x^{-1/r}]$ for some positive integer $r$ and some constant
triangular matrix $C=\left(\begin{matrix}a&0\\c&b\end{matrix}\right)$
with nonzero $a,b$ such that $A(x)T(x)=T(x^p)C$. 
For the determinants this means that $2\det(T(x))=ab\det(T(x^p))$; hence $\det(T(x))$ is a constant
and $ab=2$.
Considering the upper right elements of both sides implies $T_{12}(x)=T_{12}(x^p)b$. 

If $T_{12}(x)\neq 0$, we obtain $b=1$ and thus
$a=2$. This means that $C$ has the distinct eigenvalues $1,2$ and hence
with an additional transformation, we can replace it by $\diag(1,2)$ and
as a consequence, $T(x)$ is replaced by some lower triangular matrix.
If $T_{12}(x)=0$, we immediately obtain that $a=1,\,b=2$ and hence we can replace
$C$ by $\diag(1,2)$ also in this case.

Thus we can assume in both cases that $C=\diag(1,2)$ and that $T(x)$ is lower triangular and hence its diagonal elements
are constants. 

Now put $T(x)=\left(\begin{matrix}c&0\\u(x)&d\end{matrix}\right)$. Then $u(x)\in E$ satisfies
$cx^{-1}+2u(x)=u(x^p)$. We now consider the vector space $F$ of all formal Laurent series
$\sum_{j=-\infty}^{\infty}a_jx^{-j/r}$. Then $u(x)\in F$ and also 
$v(x):=-c\sum_{k=0}^{\infty}2^{-k-1}x^{-p^k}\in F$. As $v(x)$ satisfies $cx^{-1}+2v(x)=v(x^p)$,
the difference $d(x)=u(x)-v(x)\in F$ satisfies $2d(x)=d(x^p)$. In view of the series
expansions, this means that $d(x)=:d$ is a constant. We obtain that $u(x)=v(x)+d$
in contradiction to $u(x)\in E$.\vspace{2mm}

Fortunately, our matrices $A(x)$ are very special.
The proof of Theorem \ref{thm2-1} in  case 2M is complete, once we have shown
\begin{prop}\label{2mregsing}Consider $A(x),B(x)\in\GL_n(\hat K)$, $\hat K=\bigcup_{r\in\NX^*}C[[x^{1/r}]][x^{-1/r}]$ satisfying
the consistency condition (\ref{2mcons}). Then $x=0$ is a regular singular point of $y(x^p)=A(x)y(x)$.
\end{prop}
\begin{prf} We will show later
\begin{lemma}\label{2mtriang}
Under the hypotheses of Proposition \ref{2mregsing}, there exists a gauge transformation $y=H(x)z$, $H(x)\in\GL_n(\hat K)$, that
changes (\ref{2msys}) into 
\begin{equation}\label{2msystilde}z(x^p)=\tilde A(x)z(x),\ \ z(x^q)=\tilde B(x)z(x),\end{equation}
where $\tilde A(x)=d\,I$ with some $d\in C$
or $\tilde A(x),\tilde B(x)$ are both lower block triangular with blocks of the same size,
\ie\ there exist $m\in\{1,...,n-1\}$ and $A_{11}(x),B_{11}(x)\in\GL_{m}(\hat K)$, $A_{22}(x),B_{22}(x)\in\GL_{n-m}(\hat K)$,
$A_{21}(x),B_{21}(x)\in\hat K^{n-m,m}$ such that
\begin{equation}\label{2mblock}
\tilde A(x)=\left(\begin{matrix}A_{11}(x)&0\\A_{21}(x)&A_{22}(x)\end{matrix}\right),\ \ 
\tilde B(x)=\left(\begin{matrix}B_{11}(x)&0\\B_{21}(x)&B_{22}(x)\end{matrix}\right).
\end{equation}
\end{lemma}
{Recall from the beginning of section \ref{Sec2a} that gauge transformations preserve the consistency condition. Hence
it also holds for $\tilde A(x)$ and $\tilde B(x)$.}

We now prove the Proposition by induction. In  case $n=1$, it has been shown in Proposition 23 of
\cite{Ro15} that $y(x^p)=A(x)y(x)$ is always regular singular. Indeed, if $A(x)=a x^s b(x)$, where
$a\in C^*$, $s\in\QX$ and $b(x)\in C[[x]]$ with $b(0)=1,$ then there is a formal series $c(x)\in C[[x]]$ with $c(0)=1$
satisfying $c(x^p)=b(x)c(x)$ as a power series expansion readily shows. Thus $y(x^p)=A(x)y(x)$ is equivalent
to $z(x^p)=a x^s z(x)$ and by $z=x^{\frac s{p-1}}v$, this is equivalent to $v(x^p)=a\, v(x)$.

So suppose that the statement has been proved for all dimensions smaller than $n$. 
Given $A(x),B(x)$ as in the hypothesis, we now invoke Lemma \ref{2mtriang}.
{When $\tilde A(x)=dI$}, there is nothing to do. 
Otherwise, we observe that the couples $A_{11}(x),B_{11}(x)$ {and} $A_{22}(x),B_{22}(x)$
also satisfy the consistency condition and therefore by the induction hypothesis, $x=0$
is a regular singular point of $u(x^p)=A_{jj}(x)u(x)$, $j=1,2$. Thus there exists gauge transformations
$u=F_{jj}(x)v$ {with entries in $\hat K$} that transform the systems to constant ones $v(x^p)=\tilde A_{jj}v(x)$, $j=1,2$.
Performing the same gauge transformations on the systems $u(x^q)=B_{jj}(x)u(x)$, $j=1,2$ yields
systems $v(x^q)=\tilde B_{jj}(x)v(x)$ with $\tilde B_{jj}(x)$ having coefficients in $\hat K$.
As we still have the consistency condition, we have 
$$\tilde A_{jj}\tilde B_{jj}(x)=\tilde B_{jj}(x^p)\tilde A_{jj},\ \ j=1,2.$$
As used before, this implies that {the} $\tilde B_{jj}(x)$ are also constant and commute with $\tilde A_{jj}$.
Hence using the block diagonal matrix $F(x)=\diag(F_{11}(x),F_{22}(x))$, the system (\ref{2msystilde})
is reduced to one, where additionally $A_{jj},\ B_{jj}$ are constant and commute.

It remains to show that for a system (\ref{2msystilde}), (\ref{2mblock}) with constant commuting $A_{jj},\ B_{jj}$
satisfying the consistency condition, the point $x=0$ is a regular singular point of the 
first equation of (\ref{2msystilde}).
Observe that the consistency condition reduces to an equation for the lower left block
\begin{equation}\label{2m21}
A_{21}(x^q)B_{11}+A_{22}B_{21}(x)=B_{21}(x^p)A_{11}+B_{22}A_{21}(x).
\end{equation}
This suggests {splitting} $A_{21}(x),B_{21}(x)$ into their polar parts $A^-_{21}(x),B^-_{21}(x)$
and their regular parts $A^+_{21}(x),B^+_{21}(x)$
containing the terms with {negative or non-negative exponents, respectively}. Then  (\ref{2m21})
splits into two equations
\begin{equation}\label{2m21pm}
A^{\pm}_{21}(x^q)B_{11}+A_{22}B^{\pm}_{21}(x)=B^{\pm}_{21}(x^p)A_{11}+B_{22}A^{\pm}_{21}(x).
\end{equation}
We are interested in the polar parts that have to be removed. 
They are polynomial in $x^{-1/r}$ for some positive integer $r$.
So we replace $x=\xi^{-r}$ and consider the matrices
\begin{equation}\label{2mblock-}
A^-(\xi)=\left(\begin{matrix}A_{11}&0\\A^-_{21}(\xi^{-r})&A_{22}\end{matrix}\right),\ \ 
B^-(\xi)=\left(\begin{matrix}B_{11}&0\\B_{21}(\xi^{-r})&B_{22}\end{matrix}\right).
\end{equation}
These matrices are polynomial in $\xi$ and satisfy the consistency condition because of
(\ref{2m21pm}). Here Corollary \ref{2mpoly}
applies and yields a polynomial matrix $G(\xi)$  with $G(0)=I$ 
such that 
\begin{equation}\label{a-g}A^-(\xi)G(\xi)=G(\xi^p)A^-(0).\end{equation} 
{Writing $G(\xi)=\left(\begin{matrix}G_{11}(\xi)&G_{12}(\xi)\\
G_{21}(\xi)&G_{22}(\xi)\end{matrix}\right)$ with blocks of the same size as those of $A^-$, we find
$A_{11}G_{12}(\xi)=G_{12}(\xi^p)A_{22}$ from the $(1,2)$-blocks in (\ref{a-g}).
Expanding $G_{12}(\xi)$ in powers of $\xi$ and using $G_{12}(0)=0$, we find that $G_{12}(\xi)$ vanishes.
Similar, the diagonal blocks in (\ref{a-g}) now show that $G_{11}(\xi)=I_m$, $G_{22}(\xi)=I_{n-m}$.
Finally, its $(2,1)$-blocks show that
$A^-_{21}(\xi^{-r})+A_{22}G_{21}(\xi)=G_{21}(\xi^p)A_{11}$.}
Then the transformation $z=G(x^{-1/r})w$ changes $z(x^p)=\tilde A(x)z(x)$ into
$w(x^p)=\check A(x)w(x)$, where $G(x^{-p/r})\check A(x)=\tilde A(x)G(x^{-1/r})$. Using the block structure of
$\tilde A(x)$ and $G(\xi)$, this implies that $\check A(x)=\left(\begin{matrix}A_{11}&0\\\check A_{21}(x)&A_{22}\end{matrix}\right) $
where 
$$\check A_{21}(x) = A_{21}(x)+A_{22}G_{21}(x^{-1/r})-G_{21}(x^{-p/r})A_{11}=A_{21}(x)-A^-_{21}(x)=A^+_{21}(x).$$

Thus we have shown that $y(x^p)=A(x)y(x)$ is equivalent over $\hat K$ to $w(x^p)=\check A(x)w(x)$ where all entries of the coefficient
matrix $\check A(x)$ have positive valuation. As shown in the beginning of the proof of
Corollary \ref{2mpoly} and stated in Proposition 34 of \cite{Ro15}, 
this implies that $x=0$ is a regular singular point of the latter system and hence also of
the given system $y(x^p)=A(x)y(x)$. This completes the proof of Proposition \ref{2mregsing}.
\end{prf}

{\bf Proof of Lemma \ref{2mtriang}.} By Theorem 24 of \cite{Ro15}, there exists a gauge transformation $y=T(x)z$, $T(x)\in\GL_n(\hat K)$ that changes
$y(x^p)=A(x)y(x)$ into $z(x^p)=A^{(1)}(x)z(x)$ where 
$A^{(1)}(x)\in\GL_n(\hat K)$ is lower triangular and has constant
diagonal entries. Performing the same gauge transformation on the second
equation $y(x^q)=B(x)y(x)$, we can assume without loss of generality that
the given matrix $A(x)$ additionally has this property.
In the sequel, we denote the entries of $A(x)$ by $a_{jk}(x)$, those of $B(x)$ by $b_{jk}(x)$, $j,k=1,...,n$, and, {without mentioning it again}, 
do the same for other matrices. 
We omit the argument ``$(x)$'' if an entry is constant.

If the diagonal entries $a_{jj},$ $j=1,...,n$, were distinct, it could be shown from the 
consistency condition that $B(x)$ is also lower triangular thus proving the Lemma --
in fact, this would first be done for $b_{1n}(x)$ as the consistency condition implies that 
$a_{11}b_{1n}(x)=b_{1n}(x^p)a_{nn}$ and 
thus $b_{1n}(x)=0$ if $a_{11}\neq a_{nn}$; then the same follows successively for the 
other entries of $B(x)$ above the diagonal in a similar way.
Unfortunately, we have no information about these diagonal entry.

Dividing $A(x)$ by $a_{nn}$, we can assume that $a_{nn}=1$.
Let $m\geq1$ be the maximal length of a block $I_m$ in the lower right corner of $A(x)$, \ie\ we start with
$$A(x)=\left(\begin{matrix}A_{11}(x)&0\\A_{21}(x)&I_{m}\end{matrix}\right),\ A_{11}(x)\mbox{ lower triangular with constant diagonal. }$$
If all entries $b_{jk}(x)$ vanish for $j=1,...n-m,\ k=n-m+1,...,n$, then $B(x)$ is lower block triangular with
blocks of the same size as $A(x)$ and the lemma is proved.
Otherwise there is an entry $b_{r\ell}(x)\neq0$ with $1\leq r\leq n-m$, $n-m<\ell\leq n$. By going over to the
uppermost nonzero entry in the $\ell$-th column of $B(x)$, we can assume that $b_{j\ell}(x)=0$ for $j=1,...,r-1$
if $r>1$.

We express the element $p_{r\ell}(x)$ of $P(x)=A(x^q)B(x)=B(x^p)A(x)$
in two ways and find $p_{r\ell}(x)=a_{rr}b_{r\ell}(x)=b_{r\ell}(x^p)a_{\ell\ell}$.
This yields that $b_{r\ell}(x)$ is constant
and $a_{rr}=a_{\ell\ell}=1$. We denote $d=b_{r\ell}(x)$.

Consider now the following matrix $S(x)$: Its $r$-th column is $\frac1d B_\ell(x^{1/q})$, where $B_\ell(x)$
denotes the $\ell$-th column of $B(x)$; the other columns of $S(x)$ are the unit vectors $e_1,...,e_{r-1}$ {and}
$e_{r+1},...,e_{n}$. Observe that $S(x)$ is lower triangular. We now perform the gauge transformation $y=S(x)z$
on our system and obtain $A^{(2)}(x)=S(x^p)^{-1}A(x)S(x)$ and $B^{(2)}(x)=S(x^q)^{-1}B(x)S(x)$ which still
satisfy the consistency condition. $A^{(2)}(x)$ is still lower triangular and has unchanged diagonal entries and
lower right block.
The right multiplication $B(x)S(x)$ adds multiples of columns $r+1,...,n$ to the $r$-th column of $B(x)$ -- this is not
really interesting. The left multiplication  by $S(x^q)^{-1}$ substracts $\frac{1}db_{j\ell}(x)$ times row
$r$ from row $j$ of the resulting matrix for $j=r+1,...,n$. This means in particular that the $\ell$-th column
of $B^{(2)}(x)$ is a multiple of the $r$-th unit vector.

We consider now the entries $p^{(2)}_{j\ell}(x)$, $j=r+1,...,n$, of the product
$P^{(2)}(x)=$ $A^{(2)}(x^q)B^{(2)}(x)=$ $B^{(2)}(x^p)A^{(2)}(x)$. 
As the $\ell$-th columns of 
$B^{(2)}(x)$ and $A^{(2)}(x)$ are multiples of unit vectors,
we find that
$$p^{(2)}_{j\ell}(x)=a^{(2)}_{jr}(x^q)d=b^{(2)}_{j\ell}(x^p)a_{\ell\ell}=0$$
and hence $a^{(2)}_{jr}(x)=0$ for $j=r+1,...,n$.
Thus the $r$-th column of $A^{(2)}(x)$ equals the $r$-th unit vector.

If $r<n-m$ then we can exchange the $r$-th and $(r+1)$-th rows and columns of $A^{(2)}(x)$ and $B^{(2)}(x)$ and 
obtain $A^{(3)}(x),\,B^{(3)}(x)$ which still satisfy the consistency condition, $A^{(3)}(x)$ is 
lower triangular with constant diagonal and lower right block $I_m$, but now the additional
column that is a unit vector is the $(r+1)$-th column. If $r+1<n-m$ then we repeat the modification.
In this way, we obtain a system $v(x^p)=A^{(4)}(x)v(x)$, $v(x^q)=B^{(4)}(x)v(x)$ equivalent to (\ref{2msys})
over $\hat K$, where $A^{(4)}(x)$ is still lower triangular with constant diagonal but now has a
lower right block $I_{m+1}$, \ie\ its size has increased.
Thus we can start all over and, after a finite number of steps, we either reach a situation 
(\ref{2msystilde}), (\ref{2mblock}) proving the lemma
or we stop with $\tilde A(x)=I_n.$
This proves the Lemma and completes the proof of Theorem \ref{thm2-1} in  case 2M.
\hfill $\square$\\[0.1in]

{\noindent{Acknowledgement.} The authors would like to thank Boris Adamczewski for suggesting an improvement of the 
proof of Proposition \ref{2mprop} and for pointing out the article \cite{BCR13} to us, Nicholas Brisebarre for pointing out \cite{ApLac,Herm,Jolly,LuPe98,Mart}, and the referees for many useful comments and suggestions.  }\\

The work of the second author  was supported by a grant from the Simons Foundation (\#349357, Michael Singer).

\bibliographystyle{plain}
\bibliography{DDrefs}

\newcommand{\SortNoop}[1]{}\def\cprime{$'$} \def\cprime{$'$} \def\cprime{$'$}
  \def\cprime{$'$}
\begin{thebibliography}{10}

\bibitem{AuB16}
Boris Adamczewski and Jason~P. Bell.
\newblock A problem about {M}ahler functions.
\newblock {\em Ann. Sc. Norm. Super. Pisa}, to appear.
\newblock See also {\tt arXiv:1303.2019v1}.

\bibitem{Ad31}
C.~R. Adams.
\newblock Linear {$q$}-difference equations.
\newblock {\em Bull. Amer. Math. Soc.}, 37(6):361--400, 1931.

\bibitem{AS2003}
Jean-Paul Allouche and Jeffrey Shallit.
\newblock {\em Automatic sequences}.
\newblock Cambridge University Press, Cambridge, 2003.
\newblock Theory, applications, generalizations.

\bibitem{ApLac}
P.~Appell and E.~Lacour.
\newblock {\em Principes de la th\'eorie des fonctions elliptiques et
  applications}.
\newblock Gauthier-Villars, Paris, 1897.

\bibitem{Arreche15b}
C.E. Arreche.
\newblock On the computation of the difference-differential {Galois} group for
  a second-order linear difference equation.
\newblock Preprint, 2015.

\bibitem{AS16}
C.E. Arreche and M.F. Singer.
\newblock Galois groups for integrable and projectively integrable linear
  difference equations.
\newblock Preprint, 2016.

\bibitem{BJL1979}
W.~Balser, W.~B. Jurkat, and D.~A. Lutz.
\newblock A general theory of invariants for meromorphic differential
  equations. {I}. {F}ormal invariants.
\newblock {\em Funkcial. Ekvac.}, 22(2):197--221, 1979.

\bibitem{BJLII}
W.~Balser, W.~B. Jurkat, and D.~A. Lutz.
\newblock A general theory of invariants for meromorphic differential
  equations. {II}. {P}roper invariants.
\newblock {\em Funkcial. Ekvac.}, 22(3):257--283, 1979.

\bibitem{BaMa}
M.A. Barkatou and S.S. Maddah.
\newblock {Removing Apparent Singularities of Systems of Linear Differential
  Equations with Rational Function Coefficients.}
\newblock In {\em {Proceedings of the 40th International Symposium on Symbolic
  and Algebraic Computation}}, pages 53--60. {ACM Press}, 2015.

\bibitem{Becker94}
Paul-Georg Becker.
\newblock {$k$}-regular power series and {M}ahler-type functional equations.
\newblock {\em J. Number Theory}, 49(3):269--286, 1994.

\bibitem{BCR13}
Jason~P. Bell, Michael Coons, and Eric Rowland.
\newblock The rational-transcendental dichotomy of {M}ahler functions.
\newblock {\em J. Integer Seq.}, 16(2):Article 13.2.10, 11, 2013.

\bibitem{Bez94}
J.-P. B{\'e}zivin.
\newblock Sur une classe d'\'equations fonctionnelles non lin\'eaires.
\newblock {\em Funkcial. Ekvac.}, 37(2):263--271, 1994.

\bibitem{BG96}
J.-P. B{\'e}zivin and F.~Gramain.
\newblock Solutions enti\`eres d'un syst\`eme d'\'equations aux diff\'erences.
  {II}.
\newblock {\em Ann. Inst. Fourier (Grenoble)}, 46(2):465--491, 1996.

\bibitem{Bezivin00}
Jean-Paul B{\'e}zivin.
\newblock Sur les syst\`emes d'\'equations aux diff\'erences.
\newblock {\em Aequationes Math.}, 60(1-2):80--98, 2000.

\bibitem{BB92}
Jean-Paul B{\'e}zivin and Abdelbaki Boutabaa.
\newblock Sur les \'equations fonctionelles {$p$}-adiques aux
  {$q$}-diff\'erences.
\newblock {\em Collect. Math.}, 43(2):125--140, 1992.

\bibitem{BrFa96}
B.~L.~J. Braaksma and B.~F. Faber.
\newblock Multisummability for some classes of difference equations.
\newblock {\em Ann. Inst. Fourier (Grenoble)}, 46(1):183--217, 1996.

\bibitem{BH99}
N.~Brisebarre and L.~Habsieger.
\newblock Sur les fonctions enti\`eres \`a double pas r\'ecurrent.
\newblock {\em Ann. Inst. Fourier (Grenoble)}, 49(2):653--671, 1999.

\bibitem{Car12}
R.~D. Carmichael.
\newblock The {G}eneral {T}heory of {L}inear {$q$}-{D}ifference {E}quations.
\newblock {\em Amer. J. Math.}, 34(2):147--168, 1912.

\bibitem{Cobham}
Alan Cobham.
\newblock On the base-dependence of sets of numbers recognizable by finite
  automata.
\newblock {\em Math. Systems Theory}, 3:186--192, 1969.

\bibitem{DaSi}
J.~H. Davenport and M.~F. Singer.
\newblock Elementary and {L}iouvillian solutions of linear differential
  equations.
\newblock {\em J. Symbolic Comput.}, 2(3):237--260, 1986.

\bibitem{DHR15}
T.~Dreyfus, C.~Hardouin, and J.~Roques.
\newblock Hypertranscendence of solutions of {Mahler} equations.
\newblock arXiv:1507.03361 [math.AC]; to appear in {\it J.~Eur.~Math.~Soc.},
  2015.

\bibitem{DHR16}
T.~Dreyfus, C.~Hardouin, and J.~Roques.
\newblock Functional relations of solutions of $q$-difference equations.
\newblock arXiv:1603.06771 [math.NT], 2016.

\bibitem{Durand}
Fabien Durand.
\newblock Cobham's theorem for substitutions.
\newblock {\em J. Eur. Math. Soc. (JEMS)}, 13(6):1799--1814, 2011.

\bibitem{Ga}
F.~R. Gantmacher.
\newblock {\em The theory of matrices. {V}ols. 1, 2}.
\newblock Translated by K. A. Hirsch. Chelsea Publishing Co., New York, 1959.

\bibitem{HaSi08}
C.~Hardouin and M.F. Singer.
\newblock Differential {G}alois theory of linear difference equations.
\newblock {\em Math. Ann.}, 342(2):333--377, 2008.
\newblock Erratum in {\it Math. Ann.} (20011) 350:243-244, DOI
  10.1007/s00208-010-0551-1.

\bibitem{Herm}
C.~Hermite.
\newblock {\em Sur quelques applications des fonctions elliptiques}.
\newblock Gauthier-Villars, Paris, 1885.

\bibitem{Im84}
Geertrui~K. Immink.
\newblock {\em Asymptotics of analytic difference equations}, volume 1085 of
  {\em Lecture Notes in Mathematics}.
\newblock Springer-Verlag, Berlin, 1984.

\bibitem{Jolly}
Jean-Claude Jolly.
\newblock Solutions m\'eromorphes sur {$\Bbb C$} d'un syst\`eme d'\'equations
  aux diff\'erences \`a coefficients constants et \`a deux pas r\'ecurrents.
\newblock {\em Ann. Inst. Fourier (Grenoble)}, 52(2):585--622, 2002.

\bibitem{Ju78}
W.~B. Jurkat.
\newblock {\em Meromorphe {D}ifferentialgleichungen}, volume 637 of {\em
  Lecture Notes in Mathematics}.
\newblock Springer, Berlin, 1978.

\bibitem{LuPe98}
Lutz~G. Lucht and Manfred Peter.
\newblock On the characterization of exponential polynomials.
\newblock {\em Arch. Math. (Basel)}, 71(3):201--210, 1998.

\bibitem{Mart}
Nicolas Marteau.
\newblock Sur les \'equations aux diff\'erences en une variable.
\newblock {\em Ann. Inst. Fourier (Grenoble)}, 50(5):1589--1615, 2000.

\bibitem{DLMF}
Frank W.~J. Olver, Daniel~W. Lozier, Ronald~F. Boisvert, and Charles~W. Clark,
  editors.
\newblock {\em N{IST} handbook of mathematical functions}.
\newblock U.S. Department of Commerce, National Institute of Standards and
  Technology, Washington, DC; Cambridge University Press, Cambridge, 2010.
\newblock With 1 CD-ROM (Windows, Macintosh and UNIX).

\bibitem{OW15}
A.~Ovchinnikov and M.~Wibmer.
\newblock {$\sigma$}-{G}alois theory of linear difference equations.
\newblock {\em Int. Math. Res. Not. IMRN}, (12):3962--4018, 2015.

\bibitem{Poorten}
A.~J. {\SortNoop{Poorten}}van~der Poorten.
\newblock Remarks on automata, functional equations and transcendence.
\newblock {\em S\'eminaire de The\'eorie des Nombres de Bordeau (Univ. Bordeaux
  I, Talence)}, Exp.~No.~27:1083--1092, 1986-1987.

\bibitem{PuSi}
M.~{\SortNoop{Put}}van~der Put and M.F. Singer.
\newblock {\em Galois theory of difference equations}, volume 1666 of {\em
  Lecture Notes in Mathematics}.
\newblock Springer-Verlag, Berlin, 1997.

\bibitem{PuSi2003}
M.~{\SortNoop{Put}}van~der Put and M.F. Singer.
\newblock {\em Galois theory of linear differential equations}, volume 328 of
  {\em Grundlehren der Mathematischen Wissenschaften [Fundamental Principles of
  Mathematical Sciences]}.
\newblock Springer-Verlag, Berlin, 2003.

\bibitem{Ramis92}
J.-P. Ramis.
\newblock About the growth of entire functions solutions of linear algebraic
  {$q$}-difference equations.
\newblock {\em Ann. Fac. Sci. Toulouse Math. (6)}, 1(1):53--94, 1992.

\bibitem{Ran92}
B.~Rand\'e.
\newblock {\em {\'Equations} fonctionnelles de {Mahler} et applications aux
  suites $p$-r\'eguli\`eres}.
\newblock PhD thesis, Univ.\ de Bordeaux, 1992.
\newblock {\tt https://tel.archives-ouvertes.fr/tel-01183330}.

\bibitem{Ro15}
J.~Roques.
\newblock On the algebraic relations between {Mahler} functions.
\newblock to appear in {\it Trans.\ Amer.\ Math.\ Soc.}; {\tt
  https://www-fourier.ujf-grenoble.fr/$\sim$jroques/mahler.pdf}, 2015.

\end{thebibliography}

\end{document}